\documentclass{article}

\usepackage[T1]{fontenc}
\usepackage[utf8]{inputenc}
\usepackage[english]{babel}

\usepackage{authblk}
\title{Controlling the time discretization bias \\for the supremum of Brownian Motion}
\author[1]{Krzysztof Bisewski\footnote{Email: bisewski@cwi.nl}}
\author[1,2]{Daan Crommelin}
\author[2]{Michel Mandjes}
\affil[1]{\small{Centrum Wiskunde \& Informatica, Amsterdam}}
\affil[2]{\small{Korteweg de Vries Institute for Mathematics, University of Amsterdam}}

\usepackage[absolute]{textpos}
\title{Controlling the time discretization bias \\for the supremum of Brownian Motion}

\date{September 13, 2017 \vspace{-.5cm}}

\usepackage{graphicx}
\usepackage{epstopdf}
\usepackage{enumerate}
\usepackage{amsthm}
\usepackage{amsmath}
\allowdisplaybreaks[1] 
\usepackage{amsfonts}
\usepackage{subcaption}
\usepackage{dsfont}
\usepackage[hidelinks]{hyperref}
\usepackage{chngpage}
\usepackage{float}
\usepackage{listings}
\usepackage{color}
\usepackage{xfrac}

\usepackage{array}
\usepackage{booktabs}
\DeclareMathOperator*{\argmax}{arg\,max}

\newtheorem{definition}{Definition}
\newtheorem{theorem}{Theorem}
\newtheorem{proposition}{Proposition}
\newtheorem{lemma}{Lemma}

\newtheorem{algorithm}{Algorithm}
\newtheorem{remark}{Remark}
\newtheorem{corollary}{Corollary}
\newtheorem{example}{Example}

\newcommand{\limb}{\lim_{b\to \infty}}
\newcommand{\ind}{\mathds{1}}
\newcommand{\R}{\mathbb{R}}
\newcommand{\N}{\mathbb{N}}
\newcommand{\PPP}{\mathbb{P}}

\newcommand{\Exp}{\mathbb{E}}

\newcommand{\Var}{\mathbb{V}\textnormal{\textrm{ar}}}

\newcommand{\red}[1]{#1}

\newcommand{\vb}{\vspace{3.2mm}}

\begin{document}
\maketitle

\begin{abstract}\noindent
We consider the bias arising from time discretization when estimating the threshold crossing probability $w(b) := \PPP(\sup_{t\in[0,1]} B_t > b)$, with $(B_t)_{t\in[0,1]}$  a standard Brownian Motion. We prove that if the discretization is equidistant, then to reach a given target value of the relative bias, the number of grid points has to grow quadratically in $b$, as $b$ grows. \red{When considering non-equidistant discretizations (with threshold-dependent grid points), we can substantially improve on this}: we show that for such grids the required number of grid points is independent of $b$, and in addition we point out how they can be used to construct a strongly efficient algorithm for the estimation of $w(b)$. \red{Finally, 
we show how to apply the resulting algorithm for a broad class of stochastic processes; 
it is empirically shown that the threshold-dependent grid significantly outperforms its
equidistant counterpart.}
\end{abstract}

\begin{textblock}{0.8}[0.5,0.5](0.535,0.97) \hfill\small{\emph{Submitted on March 17, 2017}} \end{textblock}

\section{Introduction}

Extreme values of random processes play a prominent role in a broad range of practical problems. It is often of interest to find the tail of the distribution of the supremum of a continuous-time \red{stochastic} process $(X_t)_{t\ge 0}$ over a finite time interval. \red{In this paper  the focus is on the level crossing probability \[w(b) := \PPP\left(\sup_{t\in[0,1]}X_t > b\right).\] For many classes of processes, such as the Gaussian processes \cite{adler1990introduction}, typically no explicit expressions for  $w(b)$ are available, with Brownian Motion and the Ornstein-Uhlenbeck process being notable exceptions.  When an explicit expression for $w(b)$ is unavailable one usually resorts to using high-dimensional numerical integration and simulation-based methods, see e.g.\ \cite{genz2009computation} for further reading.}\\

For most of the available numerical methods, the underlying \red{continuous-time} process needs to be discretized in time. One chooses a certain \textit{finite grid} $T \subset [0,1]$ and then approximates $w(b)$ with $w_T(b) := \PPP\big(\sup_{t\in T}X_t > b\big)$. We note that this always leads to an underestimation, i.e., $w_T(b) \leq w(b)$. We quantify this underestimation by $\beta_T(b) := (w(b)-w_T(b))/w(b)$, the relative \emph{discretization bias}\footnote{\red{As $b\to\infty$, both $w_T(b)$ and $w(b)$ tend to $0$, so that the \textit{absolute} bias is not a meaningful accuracy measure.}}. Typically $T$ is chosen to be an \textit{equidistant grid} $T = \{\frac{1}{n},\frac{2}{n},\ldots, 1\}$ and in that case, $\beta_T(b)$ can be reduced only by changing the \textit{grid size} $n$. The finer the grid, the smaller the bias, but also, the larger the computational effort to estimate $w_T(b)$. The main drawback of using equidistant grids is that typically, to reach a given target value of the discretization bias, the grid size $n$ has to grow with the threshold $b$. In that case, for large $b$, the appropriate grid size can become so large that the computation is not feasible. Two central questions arise from these observations: How fast does $n$ have to grow in $b$? Furthermore, can we identify a more efficient family of grids?

\vb

In this paper we address these issues for standard Brownian Motion. Although in this case $w(b)$ can be computed explicitly, there are no available expressions for $\beta_T(b)$. We conduct a thorough study of the influence of the choice of the grid on the corresponding relative bias. Furthermore, we argue that exploring the \red{case} of standard Brownian Motion is a first step towards finding efficient grids for \red{a more general class of processes. We demonstrate numerically how our analysis of efficient grids for Brownian Motion leads to a useful procedure to determine efficient grids for a broad range of other processes.}

The contributions of this paper are the following. \red{(i)~The first finding can be seen as a negative result}: we show that to \textit{uniformly control}\footnote{In this context \emph{uniform control} means that for a fixed $\varepsilon>0$, we have that $\beta_T(b) < \varepsilon$ for all $b>0$; the grid $T$ can change in $b$.} the relative bias, the size $n$ of the equidistant grid must grow at least quadratically in $b$; see Theorem \ref{thm_eq} in Section \ref{s:BM_equi}. 
\red{(ii)~The second finding is that we can do much better by using a \emph{\red{threshold-dependent}} family of grids}, meaning that grid points change their location with $b$ (but the number of points does not increase). The discretization bias induced by this particular family of grids is uniformly controlled without having to increase the number of grid points; see Theorem \ref{THEorem} in Section \ref{s:BM_bb}. According to the best of the authors' knowledge, this is the first result which shows that a careful choice of the grid can drastically increase the accuracy of the discrete estimator of $w(b)$. Using \red{threshold-dependent} grids makes it feasible to estimate $w(b)$ with moderate grid sizes even for very high thresholds $b$, which would be impossible to estimate using equidistant grids. In particular, in Section \ref{s:algorithm} we present a strongly efficient algorithm for the estimation of $w(b)$ that relies on \red{threshold-dependent} grids.
\red{(iii)~In the third place, we point out how the ideas underlying our threshold-dependent grid can be used for a broad class of stochastic processes (including Gaussian processes, such as fractional Brownian Motion, and L\'evy processes); 
it is empirically shown that the threshold-dependent grid significantly outperforms its
equidistant counterpart.}

\vb

An efficient grid (both small in size and inducing a small discretization bias) is particularly relevant for situations with large $b$. In this respect, the work presented here connects to the rare event simulation literature. As $b$ approaches infinity, $w(b)$ decays exponentially to $0$ and standard simulation-based methods like Crude Monte Carlo to estimate $w(b)$ become extremely time consuming. We emphasize that rare event simulation methods commonly aim to control the sampling error, not the bias due to the discretization. \cite{adler2012efficient} develop an algorithm that is strongly efficient (with bounded relative sampling error) for estimation of $w_T(b)$ (rather than $w(b)$). We will show that combining their algorithm with the use of \red{threshold-dependent} grids provides a strongly efficient algorithm for estimation of $w(b)$.

\vb

\red{
A topic closely related to ours concerns the quantification of the difference between the supremum of the stochastic process taken over $[0,1]$ and the supremum taken over a finite grid $T\subset [0,1]$, i.e. \[\Delta(T) = \sup_{t\in[0,1]} X_t - \sup_{t\in T} X_t.\] There are several results in the literature that study the behavior of $\Delta(T)$ for standard Brownian Motion. \cite{asmussen1995discretization} shown that for the equidistant grids $T^\text{eq}_n = \{\frac{1}{n},\ldots,\frac{n}{n}\}$, $\sqrt{n}\,\Delta(T^\text{eq}_n)$ has a tight, non-degenerate weak limit, as $n\to\infty$ and \cite{janssen2009equidistant} derived an expansion for $\Exp\Delta(T^\text{eq}_n)$. For \textit{random grids} $T^\text{rnd}_n = \{U_1,\ldots,U_n\}$, where $U_1,\ldots,U_n$ are i.i.d. uniform samples on $(0,1)$, independent of the Brownian Motion $(X_t)_{t\in[0,1]}$, \cite{calvin1997aaverage} establish the weak limit of $\sqrt{n}\,\Delta(T^\text{rnd}_n)$. Finally, \cite{calvin1997baverage} proposed a class of \textit{adaptive} grids, meaning that the consecutive grid-points $t_{k+1}$ are chosen based on $((t_1,B_{t_1}),\ldots,(t_{k},B_{t_k}))$; given any $\delta>0$, an adaptive grid $T^\delta_n = \{t_1^\delta,\ldots,t_n^\delta\}$ is provided such that $n^{1-\delta/2}\Delta(T^\delta_n)$ has a weak limit.

In our study we  do not focus on the {difference $\Delta(T)$ between the values} of the maxima of the discrete and continuous-time Brownian Motion, but rather on the $\beta_T(b)$, i.e., the relative {difference between the probabilities} that these maxima lie above a certain fixed threshold.
} 

\vb

There are several approaches to tackle the discretization bias available in the literature. Arguably, the most widely applicable method is \emph{Multilevel Monte Carlo} (MLMC) \cite{giles2008multilevel}. It can be applied together with any numerical method that relies on discretization. The idea is to use several different \textit{levels of discretization} and spend less computational effort (draw less samples) at the finest levels of discretization. MLMC effectively reduces the computational effort, and the time saved can be used to produce even finer levels of discretization. 
It could be interesting to explore the combination of MLMC method together with the idea of \red{threshold-dependent} grids but further exploiting this procedure lies beyond the scope of this article.

\vb

One of the methods that aims to directly decrease the bias \red{induced by equidistant grids} is \emph{continuity correction}. Since the discrete-time approximation $w_T(b)$ is always smaller than $w(b)$, one could \textit{slightly} lower the threshold $b$ to compensate for the underestimation. \red{\cite{broadie1997continuity}, using the machinery developed in \cite{siegmund1985sequential}, proposed a way of lowering the threshold which improves the rate of convergence of the relative bias from $O(n^{-1/2})$, cf.\ Proposition \ref{prop:eq}, to $O(n^{-1})$, as the number of grid points $n$ grows large. However, in the non-Brownian case, it remains a non-trivial problem how much $b$ should be decreased. In fact, there is no direct way of making sure whether lowering $b$ decreases the absolute relative bias, as lowering $b$ by \emph{too much} leads to overcompensation and thus to an estimate that is {\it larger} than $w(b)$.} By contrast, it is straightforward to compare the bias induced by two different grids --- the larger the discrete estimator $w_T(b)$, the smaller the relative bias.

There are also several simulation-based algorithms that do not rely on pre-discretization. \cite{li2015rare} propose a strongly efficient algorithm for estimation of $w(b)$ for a large class of Gaussian processes (most prominently, processes with constant variance function). However, when the underlying process has a unique point of maximal variance (such as Brownian Motion), the algorithm requires the simulation of a random time $\tau\in[0,1]$ from a density $f(t) \propto \PPP(X_t>b)$, which becomes a rare event simulation problem when $b$ is large. While for an arbitrary process, the random discretization proposed in the algorithm requires a computational effort cubic in the number of grid points (in order to simulate a discrete Gaussian path), pre-discretization requires only quadratic effort; see the discussion in Section \ref{s:algorithm}.

\vb

This paper is organized as follows. Section \ref{s:pre_res} provides definitions, preliminaries, and develops a general intuition. In Section \ref{s:BM_equi} we introduce useful upper and lower bounds for the discretization bias (see Lemma \ref{lem:BM_LB}) and show that the number of points on the equidistant grid has to grow quadratically in the threshold $b$ in order to uniformly control the discretization bias. In Section \ref{s:BM_bb}, as an alternative to equidistant grids, we study \red{threshold-dependent} grids, which control the relative bias with a constant grid size, independently of $b$. The proofs of all lemmas and a proposition are postponed to Section \ref{s:proofs}. In Section \ref{s:algorithm} we present an algorithm by \cite{adler2012efficient}, that we use throughout the paper for producing the numerical results; combining this algorithm with the use of \red{threshold-dependent} grids yields a strongly efficient algorithm for estimation of $w(b)$, see Corollary \ref{cor:strongly_efficient}. \red{In Section \ref{s:application} we apply threshold-dependent grids developed in previous section to stochastic processes other than Brownian Motion: Brownian Motion with jumps, Ornstein-Uhlenbeck process and fractional Brownian Motion.} Lastly, in Section \ref{s:discussion} we present concluding remarks and discuss some ideas for future research of \textit{optimal grids}. In the appendices we collect various technical results used throughout the paper.

\section{Preliminary results}\label{s:pre_res}

Let $(B_t)_{t\in [0,1]}$ be a standard Brownian Motion on the time interval $[0,1]$ with $B_0 = 0$. We consider the probability of crossing a positive threshold $b$, that is
\begin{align}\label{defi_wb}
w(b) := \PPP \bigg(\sup_{t\in [0,1]} B_t > b\bigg).
\end{align}
For a standard Brownian Motion, an explicit formula for the threshold-crossing probability (\ref{defi_wb}) is known, namely $w(b) = 2\,\PPP(B_1>b)$, which follows directly using the \emph{reflection principle} (see e.g. \cite{morters2010brownian}). Given a \textit{finite grid} $T$ we define a discrete-time approximation of $w(b)$:
\begin{align}\label{defi_wbT}
w_T(b) := \PPP \bigg(\sup_{t\in T} B_t > b\bigg),
\end{align}
where $T = \{t_1, \ldots, t_n\}$ is a finite subset of the interval $[0,1]$, ordered such that $t_1<\ldots<t_n$. As we are mostly interested in choosing the grid $T$ efficiently, we define the following performance measure.
\begin{definition}\label{def:rel_bias}
Let $T$ be a finite grid on $[0,1]$, then
\begin{align*}
\beta_T(b) := \frac{w(b) - w_T(b)}{w(b)} = \PPP\Big(\sup_{t\in T} B_t < b \ \big|  \sup_{t\in [0,1]} B_t > b \Big)
\end{align*}
is called the relative bias induced by the grid $T$.
\end{definition}
The second representation of relative bias in Definition \ref{def:rel_bias} is especially intuitive. It means that the relative bias is \emph{the probability that $B_t$ stays below $b$ on the grid $T$, given that its supremum over $[0,1]$ is greater than $b$}. Notice that any grid which includes the endpoint $t=1$ will induce a relative bias no greater than $\frac{1}{2}$. Indeed, if $1\in T$, then $w_T(b) = \PPP(\sup_{t\in T} B_t > b) \geq \PPP(B_1 > b)$ and thus
\begin{align*}
\beta_T(b) = 1 - \frac{w_T(b)}{w(b)} \leq 1 - \frac{\PPP\big(B_1 > b\big)}{2\,\PPP\big(B_1 > b\big)} = \frac{1}{2}.
\end{align*}

Our objective is to accurately estimate $w(b)$ using discrete approximations $w_T(b)$, in a computationally efficient manner. Brownian Motion has continuous paths and thus it is always possible for a given $b$ to find a fine enough grid to bound the bias up to a desired accuracy. However, the computational cost of estimating $w_T(b)$ grows in the grid size and thus it might be infeasible to numerically compute $w_T(b)$ for large grids.\\
At this point, we emphasize that we are not as much interested in the behaviour of $\beta_T(b)$ for a fixed $b$ or a fixed $n$ but rather in asymptotic regimes in which $b$ and/or $n$ approach infinity. For every $b$ we allow to use a different grid so it seems natural to treat the grid as a function of threshold. For every $b$ we define a collection of grids of all possible sizes $\{T_1(b), T_2(b), \ldots \}$, where $T_n(b)$ has $n$ elements, and we denote $\beta_n(b) := \beta_{T_n(b)}(b)$. For a given \textit{family of grids} we are interested in behavior of $\beta_n(b)$ as $n$ or $b$ tend to infinity. The most straightforward choice for the family of grids is the following.

\begin{definition}\label{def:equi_grid}
The family $\{T_n\}_{n\in\N}$, where $T_n := \{t^n_1,\ldots,t^n_n\}$ with $t^n_k:=\frac{k}{n}$ is called the equidistant family of grids.
\end{definition}

Notice that the location of grid points on the equidistant grid is independent of $b$. Since the distance between neighboring points is equal to $\frac{1}{n}$, and since Brownian paths are continuous, it follows that $\beta_n(b) \to 0$, as $n \to \infty$ for any fixed $b$.
It has been established in \cite{asmussen1995discretization} that for $T_n$, equidistant grid, the difference between the continuous-time and discrete-time supremum $\varepsilon_n = \sup_{t\in[0,1]}B_t-\sup_{t\in T_n} B_t$ is of order $n^{-1/2}$. More precisely, the sequence $(\sqrt{n}\varepsilon_n)_{n\in\N}$ has a tight and non-degenerate weak limit.
\begin{proposition}\label{prop:eq}
Let $(B_t)_{t\in[0,1]}$ denote standard Brownian Motion and $\{T_n\}_{n\in\N}$ be the equidistant family of grids from Definition $\ref{def:equi_grid}$ with $\beta_n(b) := \beta_{T_n}(b)$. For any threshold $b>0$ there exist positive constants $C_1,C_2$ such that
\begin{align*}
C_1\,n^{-1/2} \leq \beta_n(b) \leq C_2\,n^{-1/2}.
\end{align*}
\end{proposition}
The proof of the Proposition \ref{prop:eq} is given in Section \ref{s:proofs}. The proof we give strongly resembles the proof of Theorem \ref{thm_eq} below in Section \ref{s:BM_equi}, but we remark that it is also possible to derive it using the tools developed in \cite{broadie1997continuity}.

Proposition \ref{prop:eq} states that $\beta_n(b)$ decays like $n^{-1/2}$, when $n$ grows large \emph{for a fixed $b$} but it does not describe the behavior of the relative bias when $b$ varies. In Theorem \ref{thm_eq} in the following section, we derive an upper bound for $\beta_n(b)$ for $n$ and $b$ simultaneously.

Figure \ref{fig:relative_bias_equi} shows the evolution of the relative bias for four different thresholds $b=5, 6, 7, 8$ against the size of the equidistant grid. Even though all four graphs show the $n^{-1/2}$ decay, the graphs rise up with growing threshold. In particular, for thresholds $b=5$ and $8$ respectively $n=700$ and $1700$ points are needed to arrive at around $10\%$ relative bias. It indicates that, as $b$ grows \emph{increasingly many grid-points are needed to arrive at the target relative bias}. Using the \red{threshold-dependent} grid that we develop in Section \ref{s:BM_bb} one can arrive at $10\%$ relative bias using approximately $n=100$ grid-points, independently of the value of the threshold. This amounts to a substantial improvement of the computational efficiency.

\begin{figure}[h]
 \begin{adjustwidth}{0cm}{}
        \centering
       \includegraphics[scale=.6]{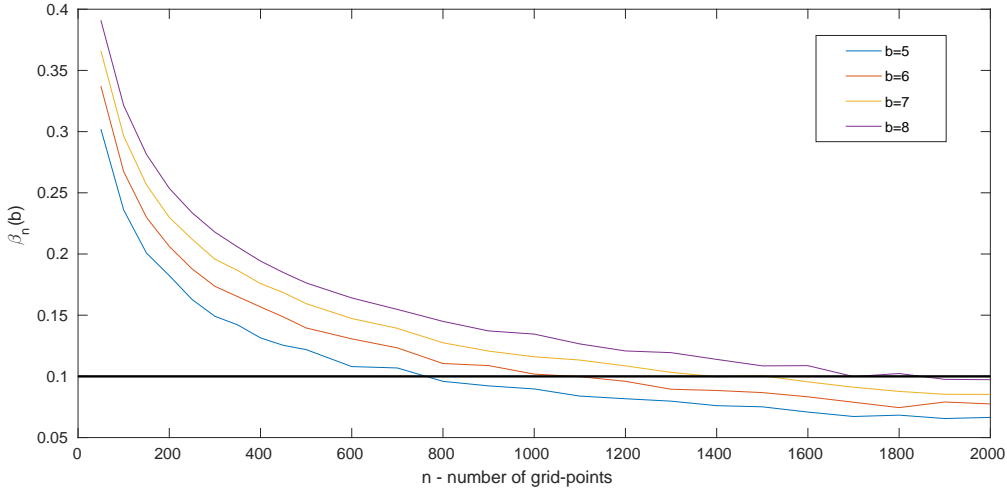}
      \caption{Plots of the relative bias $\beta_n(b)$ against the grid size $n$ for the equidistant family of grids for four different thresholds. The numerical results are computed using an algorithm described in Section \ref{s:algorithm}.
      }\label{fig:relative_bias_equi}
       \end{adjustwidth}
\end{figure}

In some cases, the equidistant family of grids is the best possible choice, in the sense that other grid families require at least equally fast asymptotic growth of $n$ as $b$ increases, in order to control the relative bias. \cite{adler2012efficient} prove that for \textit{centered, homogeneous and twice continuously differentiable (in a mean squared sense) Gaussian processes}, $n$ has to grow linearly in $b$ to uniformly control the relative bias. Moreover, if $n$ grows sublinearly in $b$, then the relative bias of any family of grids (not necessarily equidistant) tends to
its maximal value, as $b$ approaches infinity. It is noted, however, that Brownian Motion does not belong to the family of Gaussian processes for which the result of \cite{adler2012efficient} applies.\\

In the following two sections we analyze the asymptotic behavior of the relative bias $\beta_n(b)$ for two families of grids. We prove that the equidistant grid requires quadratic growth of $n$ in $b$ (see Theorem~\ref{thm_eq} in Section \ref{s:BM_equi}). As an alternative, we develop the \red{threshold-dependent} family of grids, for which we prove that the relative bias can be made arbitrarily small, uniformly in $b$ for fixed $n$ (see Theorem~\ref{THEorem} in Section~\ref{s:BM_bb}). We obtain a uniform rate of convergence in $n$ and also provide a closed-form expression for the \red{threshold-dependent} family of grids (see Definition (\ref{THE_grid1}) in Section \ref{s:BM_bb}).

\section{Equidistant family of grids for Brownian Motion}\label{s:BM_equi}

This section is devoted to analyzing the asymptotic behavior of the relative bias for the equidistant family of grids. The methodology developed in this section will be used later to prove Theorem \ref{THEorem}; in particular, the crucial part of the proof concerns bounds for the relative bias induced by an arbitrary finite grid, developed in Lemma \ref{lem:BM_LB}.\\

The following theorem describes the asymptotic behaviour of the relative bias, under the equidistant family of grids.
\begin{theorem}\label{thm_eq}
Let $(B_t)_{t\in[0,1]}$ denote standard Brownian Motion and $\{T_n\}_{n\in\N}$ be the equidistant family of grids from Definition \ref{def:equi_grid} with $\beta_n(b) := \beta_{T_n}(b)$.
\begin{enumerate}[(a)]
\item Let $b_0$ be any positive, real number. There exist positive constants $C_0,C_1$, independent of $b$ and $n$ such that
\begin{align*}
\beta_n(b) \leq C_0\cdot bn^{-1/2},
\end{align*}
for all $b \geq b_0$, and
\begin{align*}
\beta_n(b) \leq C_1 \cdot n^{-1/2},
\end{align*}
for all $b\in(0,b_0]$.
\item Let $m : (0,\infty) \to (0,\infty)$ be such that $\limb {m(b)}/{b^2} = 0$. Then, as $b\to\infty$,
\begin{align*}
\inf_{n \leq m(b)} \beta_n(b) \longrightarrow \frac{1}{2}.
\end{align*}
\end{enumerate}
\end{theorem}


\red{
Part (a) of Theorem \ref{thm_eq} states that $\beta_n(b) \leq C_0\, bn^{-1/2}$, so that in order to bound $\beta_n(b)$ uniformly in $b$ it suffices to take $n = O(b^2)$. The second part of the Theorem \ref{thm_eq} states that if $n=o(b^2)$ then $\beta_n(b)\to 1/2$, meaning that the relative bias cannot be bounded by an arbitrarily small number. Together, the two parts entail that the growth $n=O(b^2)$ is sufficient and there is no better (slower) growth which would guarantee a uniformly bounded relative bias.

\vb
}

The crucial part of the proof of Theorem \ref{thm_eq} is the method of bounding the relative bias. Since no explicit expressions for $w_T(b)$ or $\beta_T(b)$ are known (even if $T$ is an equidistant grid) we develop a general upper bound for $\beta_T(b)$ in the following lemma, in which we use the quantities
\red{\begin{align*}a_j(b) &:= \PPP\big(B_{t_j(b)-t_{j-1}(b)} < 0, \ldots, B_{t_n(b)-t_{j-1}(b)} < 0\big),\:\:\:\:
a_{n+1}(b):=\sfrac{1}{2}, \\
w_j(b)& := \PPP\big(\tau_b\in (t_{j-1}(b),t_j(b)] \ \big| \ \tau_b\in(0,1]\big),\\
\tau_b &:= \inf\{t\geq 0 : B_t > b \}. 
\end{align*}Notice that in this definition of $a_j(b)$ and $w_j(b)$ we allow grid points $t_1,\ldots,t_n$ to change their location with $b$. In the present section, which is on equidistant grids, the grid points obviously do not depend on $b$, but in later sections they do.}

\begin{lemma}\label{lem:BM_LB}
Let $T(b) = \{t_1(b), \ldots, t_n(b)\}\subset[0,1]$, where $0<t_1(b)<\ldots < t_n(b)\leq 1$, and let $t_0(b) = 0$. The following lower and upper bounds for $\beta_T(b)$ apply:
\begin{align*}
\underbar{$\beta$}_T(b) \leq \beta_T(b) \leq  \bar\beta_T(b)
\end{align*}
with
\begin{align*}
\underbar{$\beta$}_T(b) := \frac{1}{2}\sum_{j=1}^{n} a_{j+1}(b) \, w_j(b), \ \ \ \ \ \bar\beta_T(b) := \sum_{j=1}^n a_j(b) \, w_j(b).
\end{align*}
\end{lemma}

A short proof of Lemma \ref{lem:BM_LB} is included in Section \ref{s:proofs}. The bounds consist of elements of two types: $a_j(b)$, the probability that $B_t$ stays negative at times $t_j-t_{j-1}, \ldots, t_n-t_{j-1}$, and $w_j(b)$, the probability that $B_t$ hits $b$ for the first time in the interval $[t_{j-1},t_j]$ given that its supremum over $[0,1]$ is greater than $b$.\\

For a general grid \red{$T(b)$}, the \red{probabilities $a_j(b)$} are difficult to control. However, when $T(b)$ is equidistant \red{(thus independent of $b$)}, then also the probabilities $a_j$ are independent of $b$; \red{we emphasize this independence by writing $a_j$ instead of $a_j(b)$ throughout this section.} As a result,  there exists a tight asymptotic bound for them (see Lemma \ref{lem_feller} below); we were inspired to look into such quantities while reading \cite[Section 5]{morters2010brownian}. \red{The probabilities $w_j(b)$} are controlled using a mean value theorem, see Appendix \ref{appendix:results}.\ref{mvt_taub}.

\begin{lemma}\label{lem_feller}
There exist constants $C^*_1, C^*_2 > 0$ such that:
\begin{align*}
C^*_1 n^{-1/2} \leq \PPP\Big(B_1 > 0, \ldots, B_n > 0\Big) \leq C^*_2 n^{-1/2}
\end{align*}
for all $n\in\N$.
\end{lemma}
In fact, the assertion in Lemma \ref{lem_feller} is true for any symmetric random walk; \red{see \cite[Theorem~4 in Section XII.7, and Lemma~1 in Section XII.8]{feller1971introduction}}. Before proving Theorem \ref{thm_eq} we present one more lemma.
\begin{lemma}\label{equi_grid_increasing_b}
Let $T = \{t_1, \ldots, t_n\}$ be such that $t_k = \frac{k}{n}$ and let $t_0 = 0$. Then the upper bound $\bar\beta_T(b)$ developed in Lemma \ref{lem:BM_LB}  is an increasing function of $b$.
\end{lemma}
An important implication of Lemma \ref{equi_grid_increasing_b} is that for any $b_0>0$ we have that $\beta_T(b) \leq \bar\beta_T(b) \leq \bar\beta_T(b_0)$ uniformly for all $b\leq b_0$, which completely covers the statement on the situation that $b\leq b_0$ in part~(a) of Theorem \ref{thm_eq}. The proof of Lemma \ref{equi_grid_increasing_b} is given in Section \ref{s:proofs}.

\begin{proof}[Proof of Theorem \ref{thm_eq}(a)]
Thanks to Lemma \ref{equi_grid_increasing_b} it suffices to prove the first part of Theorem \ref{thm_eq}(a), i.e. we assume that $b\geq b_0$. Without loss of generality we put $b_0 = 1$. Exploiting the upper bound developed in Lemma \ref{lem:BM_LB} we decompose the sum $\sum_{j=1}^n a_j\cdot w_j(b)$ into three parts, which we treat separately:
\begin{align}\label{to_bound_thm_eq}
\beta_{n}(b) & \leq a_1\cdot w_1(b) + \sum_{j=2}^{n-1} a_j\cdot w_j(b) + a_n\cdot w_n(b),
\end{align}
Using the definition of the equidistant grid and the scaling property of Brownian motion we can see that $a_j = \PPP\big( B_{t_j-t_{j-1}} < 0, \ldots, B_{t_n-t_{j-1}} < 0 \big) = \PPP\big( B_1 < 0, \ldots, B_{n-j+1} < 0 \big)$ and the bound in Lemma \ref{lem_feller} yields $a_j \leq C^*_2 (n-j+1)^{-1/2}$ for all $j\in\{1,\ldots,n\}$. Since all $w_j(b)\leq 1$, we thus have a straightforward bound for the first term in (\ref{to_bound_thm_eq}):
\begin{align*}
a_1\cdot w_1(b) \leq C^*_2 \,n^{-1/2}
\end{align*}
The second term we bound in the following fashion, relying on the \red{upper bound} that we have for $w_j(b)$ (stated in Result B.V),
\begin{align}
\nonumber \sum_{j=2}^{n-1} a_j\cdot w_j(b) & \leq \sum_{j=2}^{n-1} C_2^*\,(n-j+1)^{-1/2} \cdot \frac{b(b+\sqrt{b^2+4})}{4} \frac{\sqrt{n}}{(j-1)^{3/2}} e^{-\frac{b^2}{2} \cdot \left( \frac{n}{j} - 1\right)} \\
\label{rieman_summ_nice} & \leq C_1 \cdot bn^{-1/2} \cdot \sum_{j=2}^{n-1} \frac{1}{n} \cdot \left( \frac{b}{\sqrt{1-\frac{j}{n}}} \cdot \left(\frac{j}{n}\right)^{-3/2} e^{-\frac{b^2}{2} \cdot \left( \frac{n}{j} - 1\right)} \right) \\
\label{rieman_b0small} & \leq C_1 \cdot bn^{-1/2} \cdot \int_0^1 \frac{b}{\sqrt{1-x}} \cdot x^{-3/2} \cdot e^{-\frac{b^2}{2}\, (1/x - 1)}\,dx \\
\nonumber & \leq C_1 \cdot bn^{-1/2},
\end{align}
where $C_2^*$ comes from Lemma \ref{lem_feller} and $C_1$ is a positive constant, independent of $b$ and $n$. 
To arrive at~(\ref{rieman_summ_nice}) we use that $2(j-1) \geq j$ for all relevant $j$. In the transition from (\ref{rieman_summ_nice}) to (\ref{rieman_b0small}) we use the definition of the Riemann sum for the function 
\[f(b,x) := \frac{b}{\sqrt{1-x}}x^{-3/2}e^{-\frac{b^2}{2}(1/x-1)};\] note that, since $f(b,x)$ is an increasing function of $x$ when $b\geq 1$ (see Result \ref{appendix:results}.\ref{res_increasing} in the Appendix), the Riemann sum in (\ref{rieman_summ_nice}) underestimates the integral, i.e., $\sum_{j=2}^{n} \frac{1}{n}f(b,\frac{j-1}{n}) \leq \int_0^1 f(b,x)\,dx = \sqrt{2\pi}$.\\
Lastly, since $a_n = \PPP(B_{t_n}<0) = \frac{1}{2}$ we have
\begin{align*}
a_n\cdot w_n(b) \leq \frac{1}{2} \cdot \frac{b(b+\sqrt{b^2+4})}{4} \frac{\sqrt{n}}{(n-1)^{3/2}} \leq C_2 \,\frac{b^2}{n},
\end{align*}
where $C_2$ is a positive constant independent of $n$ and $b$. Since $w_n(b) \leq 1$ this results in \[a_n\cdot w_n(b) \leq \min \left\{ C_2 \,\frac{b^2}{n},\frac{1}{2}\right\} \leq \sqrt{\min \left\{ C_2 \,\frac{b^2}{n},\frac{1}{2}\right\}} \leq \sqrt{C_2}\,bn^{-1/2}.\] 
Combining the above bounds,
\begin{align*}
\beta_{n}(b) & \leq a_1\cdot w_1(b) + \sum_{j=2}^{n-1} a_j\cdot w_j(b) + a_n\cdot w_n(b) \leq C_2^* n^{-1/2} + C_1 \,bn^{-1/2} + \sqrt{C_2} \,bn^{-1/2}\\
& \leq C_0 \, bn^{-1/2},
\end{align*}
where $C_0$ is a positive constant independent of $b$ and $n$. This concludes the proof.
\end{proof}

\begin{proof}[Proof of Theorem \ref{thm_eq}(b)] Without loss of generality we can assume $m(b)\to\infty$ as $b\to\infty$. Similar to the proof of Lemma \ref{lem:BM_LB} in Section \ref{s:proofs} we obtain:
\begin{align}
\nonumber w(b)\,\beta_T(b) & = \PPP\Big(\sup_{t\in T} B_t < b, \sup_{t\in [0,1]} B_t > b \Big) = \PPP\Big(\sup_{t\in T} B_t < b, \tau_b\in[0,1] \Big) \\
\nonumber & = \sum_{j=1}^n \PPP\Big(\sup_{t\in \{t_j, \ldots, t_{n}\}} B_t < b, \tau_b\in(t_{j-1},t_j] \Big) \geq \PPP\Big( B_{t_n} < b, \tau_b\in(t_{n-1},t_n] \Big)\\
\nonumber & = \int_{t_{n-1}}^{t_n} \PPP(B_{t_n} < b | B_s = b) \PPP(\tau_b\in ds) = \frac{1}{2}\,\PPP(\tau_b\in(t_{n-1},t_n))
\end{align}
Dividing both sides of the inequality by $w(b)$ yields an elementary lower bound on $\beta_T(b)$:
\begin{align}\label{very_simple_lower_bound}
\beta_T(b) & \geq \frac{1}{2}\,\PPP\big(\tau_b\in (t_{n-1},t_n] \ \big| \ \tau_b\in(0,1]\big) = \frac{1}{2} \cdot \frac{\Phi(-b/\sqrt{t_n}) - \Phi(-b/\sqrt{t_{n-1}})}{\Phi(-b)}
\end{align}
\red{where $\Phi(\cdot)$ denotes the standard normal cdf, and we use the fact that $\PPP(\tau_b\leq t) = 2\,\PPP(B_t>b)$.} In our case $t_n = 1$ and $t_{n-1} = \frac{n-1}{n} \leq \frac{m-1}{m}$, so that \red{due to the monotonicity of $\Phi(\cdot)$}
\begin{align}\label{ineq_funna}
\inf_{n\leq m(b)}\beta_n(b) \geq \frac{1}{2} - \frac{1}{2} \frac{\Phi(-b/\sqrt{(m-1)/m})}{\Phi(-b)} \, .
\end{align}
Taking the limit $b\to\infty$ on both sides of inequality (\ref{ineq_funna}) yields:
\begin{align}
\limb \nonumber \inf_{n\leq m(b)}\beta_n(b) & \geq \frac{1}{2} - \frac{1}{2}\limb
 \frac{\Phi(-b/\sqrt{(m-1)/m})}{\Phi(-b)} \\
\label{using_some_limit_result} & = \frac{1}{2} - \frac{1}{2}\limb \frac{\frac{\sqrt{(m-1)/m}}{b}\phi(b/\sqrt{(m-1)/m})}{\frac{1}{b}\phi(b)} \\
\nonumber & = \frac{1}{2} - \frac{1}{2}\limb e^{-b^2/(2(m-1))} = \frac{1}{2}
\end{align}
where \red{$\phi(\cdot)$ denotes the standard normal pdf}, we use result \ref{appendix:results}.\ref{limit1} in (\ref{using_some_limit_result}), and the last equality is a consequence of the assumption that $\limb {m(b)}/{b^2} = 0$.
\end{proof}

In this section we have proven that in order to uniformly control the relative bias, the size of the equidistant grid must grow at least quadratically in $b$, as $b$ approaches infinity. In the next section we present a \red{threshold-dependent} grid, which yields a uniform bound on the relative bias using a grid of given size. In other words, in order to control the relative bias with increasing $b$, instead of adding more and more points to the grid, it suffices to suitably shift their location.

\section{\red{Threshold-dependent} grids for Brownian Motion}\label{s:BM_bb}

In this section we prove the main result of the paper. We explicitly present a \red{threshold-dependent} family of grids which uniformly (in $b$) bounds the relative bias.\\
Before we introduce the result, we give some intuition why it is possible to control the relative bias as $b$ grows, without increasing $n$. \red{Firstly, for any given  $\varepsilon>0$, we have that \[\PPP(\sup_{t\in[0,1-\varepsilon]}B_t>b) = 2\,\PPP(B_{1-\varepsilon}>b) = o(w(b)),\]as $b\to\infty$. Therefore,} \[\red{\PPP\Big(\sup_{t\in[1-\varepsilon,1]} B_t > b \,\Big| \sup_{t\in[0,1]} B_t > b\Big) \longrightarrow 1, \ \text{ as } b \to \infty.}\]
It means that with growing $b$, the `hitting of the threshold' occurs closer and closer to  time $t=1$. It indicates that the grid points should be gradually shifted towards the point $t=1$, as $b$ is increasing. Moreover, the result in Theorem \ref{thm_eq} indicates how fast the points should be shifted. It states that for the family of equidistant grids, the uniform bound on the bias is achieved if the number of grid points grows quadratically in $b$. Equivalently, the distances between neighboring points are decreasing proportionally to $b^{-2}$. It turns out that this is indeed the pace at which the points should be shifted towards $t=1$.

In the following result,  $\Phi(\cdot)$ and $\Phi^{-1}(\cdot)$ denote the standard normal cdf and its inverse, respectively. 
\begin{theorem}\label{THEorem}
Let $(B_t)_{t\in[0,1]}$ be a standard Brownian Motion. Fix $b_0 > 0$ and let $\{T_n(b)\}_{n\in\N,b>0}$ be a family of grids such that $T_n(b) = \{t^n_1(b), \ldots, t^n_n(b)\}$; here $t^n_k(b) := \frac{k}{n}$ for $b\leq b_0$, and
\begin{align}\label{THE_grid1}
t^n_k(b) := \Bigg(\frac{b}{\Phi^{-1}\left(\frac{k}{n}\,\Phi(-b)\right)} \Bigg)^{2},
\end{align}
for $b>b_0$. Denote $\beta_n(b) := \beta_{T_n(b)}(b)$. There exists a positive $C$, independent of $b$ and $n$, such that
\begin{align*}
\beta_n(b) \leq C\,n^{-1/4}
\end{align*}
for all $b>0$.
\end{theorem}

We emphasize that the bound for the relative bias $\beta_n(b)$ developed above does not depend on the threshold $b$ and thus holds \emph{uniformly}, for all $b$. Figure \ref{fig:relative_bias_equi_vs_adapt} shows the comparison between the relative bias of the equidistant and the \red{threshold-dependent} grid, both of size $n=100$. The bias induced by the \red{threshold-dependent} grid remains uniformly bounded (by circa $0.1$), while the former tends to $0.5$, the worst possible relative bias, cf.\ Theorem \ref{thm_eq}, part~(b).\\

\begin{figure}[h]
 \begin{adjustwidth}{0cm}{}
        \centering
       \includegraphics[scale=.8]{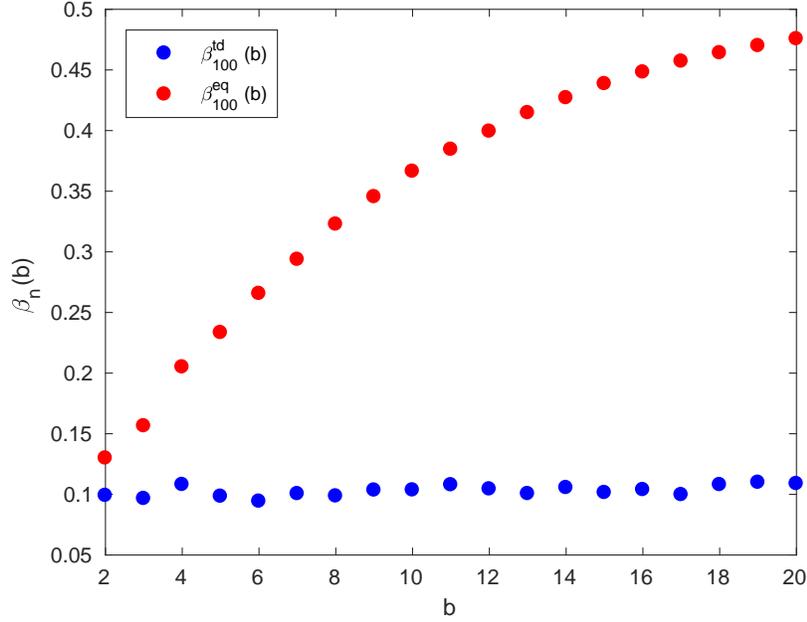}
      \caption{A plot of the relative bias of the equidistant grid $\beta_{100}^{\text{eq}}(b)$ and the \red{threshold-dependent} grid \red{$\beta_{100}^{\text{td}}(b)$}, both with fixed grid size $n=100$, as a function of $b$. Notice that $\beta_{100}^\text{eq}(b)$ tends to 0.5 with growing $b$ (the largest possible bias), while \red{$\beta_{100}^\text{td}(b)$} remains bounded by about $0.1$. The numerical results are computed with the algorithm discussed in Section \ref{s:algorithm}. The relative error due to finite sample size is negligible (smaller than 0.006).}\label{fig:relative_bias_equi_vs_adapt}
       \end{adjustwidth}
\end{figure}

Notice that for small $b$, $\{T_n(b)\}_{n\in\N,b\in(0,b_0]}$ in Theorem \ref{THEorem} is identical to the equidistant family of grids. In fact this is exactly the setting of \red{the second part of the Theorem \ref{thm_eq}(a)}. The real contribution of Theorem \ref{THEorem} is the regime when $b>b_0$. The grid defined in (\ref{THE_grid1}) is the unique solution to the set of equations
\begin{align}\label{THE_grid2}
\PPP\Big(\tau_b \in(t^n_{k-1}(b),t^n_k(b)] \ \Big| \ \tau_b \in (0,1]\Big) = \frac{1}{n}
\end{align}
for all $k \in \{1, \ldots, n\}$ and $t_0 := 0$. To see this, we sum up the first $k$ equations in (\ref{THE_grid2}) and obtain an explicit equation for $t_k^n(b)$:
\begin{align}\label{THE_grid3}
\PPP\Big(\tau_b \in(0,t^n_k(b)] \ \Big| \ \tau_b \in (0,1]\Big) = \frac{k}{n}.
\end{align}
\red{Since for Brownian Motion it holds that 
\[\PPP(\tau_b \in(0,t^n_k(b)]) = 2\,\PPP(B_{t^n_k(b)}>b) = 2\,\Phi(-b/\sqrt{t^n_k(b)}),\] and in particular $\PPP(\tau_b \in (0,1]) = 2\,\PPP(B_1>b) = 2\,\Phi(-b)$, Eqn.\ (\ref{THE_grid3}) can be equivalently expressed as
\begin{align}\label{THE_grid4}
\frac{\PPP(B_{t^n_k(b)}>b)}{\PPP(B_1>b)} = \frac{k}{n}
\end{align}
or, in terms of the cdf $\Phi(\cdot)$,
\begin{align*}
\Phi\left(-b/\sqrt{t^n_k(b)}\right) = \frac{k}{n}\Phi(-b).
\end{align*}
}

\noindent
Finally, after taking the inverse $\Phi^{-1}(\cdot)$ from both sides of the equation above we see that $t_k^n(b)$ satisfies (\ref{THE_grid1}). Figure \ref{fig:grid_evolution} shows the placement of the grid-points on the grid $T_{5}(b)$, as defined in (\ref{THE_grid1}), for increasing $b$. In fact, one can prove that
\begin{align}\label{limit_adaptive_grid}
b^2\big(1-t^n_k(b)\big) \xrightarrow{b\to\infty} -2\log(k/n)
\end{align}
and thus \[t^n_k(b)\approx 1 - \frac{2\log(n/k)}{b^2}\] for large $b$. It means that the points of the grid (\ref{THE_grid1}) are clustered around $t=1$, with distances between the points proportional to $b^{-2}$. Here we see \red{an important connection with Theorem \ref{thm_eq}(a)}, where the distances between grid-points decrease at the same pace, as already mentioned in the opening paragraph of this section.

\begin{figure}[h]
 \begin{adjustwidth}{0cm}{}
        \centering
       \includegraphics[scale=.7]{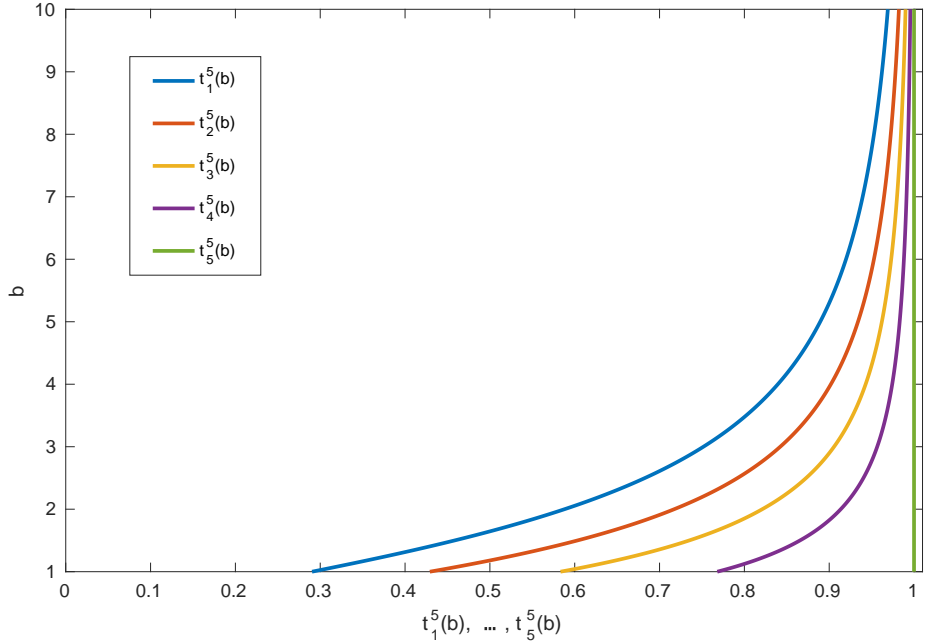}
      \caption{Location of the grid-points $t^{5}_1(b), \ldots, t^{5}_{5}(b)$ defined in (\ref{THE_grid1}) with increasing threshold $b$. Note that with growing $b$ all the points are gradually shifted towards the end-point $t=1$.}\label{fig:grid_evolution}
       \end{adjustwidth}
\end{figure}

For $b>b_0$, the points $t^n_1(b), \ldots, t^n_n(b)$ of the \red{threshold-dependent} grid (\ref{THE_grid1}) do not coincide with the equidistant grid, entailing that we can not directly use Lemma \ref{lem_feller} to control the terms of type $a_j(b)$ in the upper bound developed in Lemma \ref{lem:BM_LB} in Section \ref{s:BM_equi}. The following lemma resolves this issue.

\begin{lemma}\label{lemma2_me}
Let $t_0 = 0 < t_1 < t_2 < \ldots < t_n < \infty$. Then
\begin{align*}
\PPP\Big(B_{t_1} > 0, \ldots, B_{t_n} > 0\Big) \leq \PPP\Big(B_1 > 0, \ldots, B_N > 0\Big),
\end{align*}
for any $N \leq N_n$, where
\[N_n:=\left(\frac{t_n}{\max_{k=1,\ldots,n}(t_k - t_{k-1})}\right)^{1/2}.\]
\end{lemma}

A proof of this lemma is provided in Section \ref{s:proofs}. Lemma \ref{lem_feller} applied to the upper bound in Lemma~\ref{lemma2_me} yields a simple upper bound for $\PPP\big(B_{t_1} > 0, \ldots, B_{t_n} > 0\big)$ for any choice of $t_0 = 0 < t_1 < t_2 < \ldots < t_n < \infty$. In our case, after applying Lemma \ref{lem:BM_LB} we have to control probabilities of the type $\PPP\big( B_{t_j-t_{j-1}} < 0,\ldots, B_{t_n-t_{j-1}} < 0 \big)$, and thus we need a lower bound on \[\frac{t_n-t_{j-1}}{\max_{k=j,\ldots,n}(t_k-t_{k-1})},\] which we give in the following lemma.

\begin{lemma}\label{lemming} For the grid in (\ref{THE_grid1}), for $k>j$, $b>0$ and $n\in\N$ we have:
\begin{align*}
\textnormal{(a)} \ \ \ \ \frac{t_n^n(b) - t_j^n(b)}{t_k^n(b) - t_j^n(b)} \geq \frac{\log n - \log j}{\log k - \log j}.
\end{align*}
and when additionally $b\geq\sqrt{3}$ we have
\begin{align*}
\textnormal{(b)} \ \ \ \ \max_{k=j,\ldots,n} (t^n_{k}(b)-t^n_{k-1}(b)) = t^n_{j}(b)-t^n_{j-1}(b)
\end{align*}
\end{lemma}
Lemma \ref{lemming} is proven in Section \ref{s:proofs}. The lower bound in part (a) of Lemma \ref{lemming} is in fact \[\lim_{b\to\infty}\frac{1 - t_j^n(b)}{t_k^n(b) - t_j^n(b)}.\] With these lemmas we can prove Theorem \ref{THEorem}.

\begin{proof}[Proof of Theorem \ref{THEorem}] Part (a) of Theorem \ref{thm_eq} states that for any choice of $b_0$ there exists positive $C_1$ such that $\beta_n(b) \leq C_1n^{-1/2}$ for $b\leq b_0$ and thus also $\beta_n(b) \leq C_1n^{-1/4}$. Without the loss of generality, from now on we assume that $b>b_0 = \sqrt{3}$. Fix $n\in\N$ and denote $t_k := t^n_k(b)$ for notational simplicity. After combining the general upper bound from Lemma \ref{lem:BM_LB} with the equivalent definition (\ref{THE_grid2}) of the \red{threshold-dependent} grid (\ref{THE_grid1}) we obtain
\begin{align*}
\beta_n(b) & \leq \sum_{j=1}^n a_j(b)w_j(b) = \frac{1}{n} \sum_{j=1}^n a_j(b);
\end{align*}
observe that in our setting $w_n(b)=\frac{1}{n}.$
Moreover, Lemma \ref{lemma2_me} yields (recalling the definition of $a_n(b)$)
\begin{align*}
\beta_n(b) \leq \frac{1}{2n} + \frac{1}{n} \sum_{j=2}^{n-1} \PPP\big( B_1 > 0, \ldots, B_{N_n(j)} > 0 \big) + \frac{1}{2n},
\end{align*}
where \[N_n(j) := \left[ \left(\frac{t_n^n(b) - t_{j-1}^n(b)}{\max_{k\geq j}|t^n_k(b) - t^n_{k-1}(b)|}\right)^{1/2}\right].\] Combining Lemma \ref{lem_feller} with Lemma \ref{lemming} gives
\begin{align*}
\beta_n(b) \leq \frac{1}{n} + C\frac{1}{n} \sum_{j=2}^{n-1} \widetilde{N}_n(j)^{-1/4}, \ \ \text{ where } \ \widetilde{N}_n(j) :=  \frac{\log n - \log (j-1)}{\log j - \log (j-1)}
\end{align*}
with a constant $C>0$ that is independent of $n$ and $b$. Notice that $\widetilde{N}_n(j)$ does not depend on $b$. For $b>b_0$ we thus obtain
\begin{align}
\nonumber \beta_n(b) & \leq \frac{1}{n} + C\,n^{-1} \sum_{j=2}^{n-1} \left(\frac{\log j-\log(j-1)}{\log n-\log(j-1)}\right)^{1/4} \\
\nonumber & = \frac{1}{n} + C\,n^{-1}  \sum_{j=2}^{n-1} \left(\frac{\log\big(1 + \frac{1}{j-1}\big)}{\log\big(\frac{n}{j-1}\big)}\right)^{1/4} \\
\label{log_nierownosc} & \leq \frac{1}{n} + C\,\,n^{-1/4}  \sum_{j=2}^{n-1}\frac{1}{n}\left(\frac{\frac{n}{j-1}}{\log\big(\frac{n}{j-1}\big)}\right)^{1/4} \\
\label{rieman_b0big} & \leq \frac{1}{n} + C\,\,n^{-1/4}  \int_0^1 \left( \frac{1}{-x\log x} \right)^{1/4}\,dx \\
\nonumber & \leq C\,n^{-1/4}
\end{align}
where $C$ is a constant, independent from $n$ and $b$, that might differ from line to line. In (\ref{log_nierownosc}) we use the inequality $\log(1+x)\leq x$ and in (\ref{rieman_b0big}) we use the convergence of the Riemann sum to the integral. This concludes the proof of Theorem \ref{THEorem}.
\end{proof}

\red{
\begin{remark}{\rm
For the purpose of showing that for any confidence level $\alpha$ and bias $\varepsilon$, \red{see also (\ref{confidence}),} the `equiprobable' grid (as defined through (\ref{THE_grid1})) requires a   computational effort that is bounded in $b$, it suffices that the decay of the upper bound for $\beta_n(b)$ in Theorem \ref{THEorem} is of order $n^{-1/4}$; \red{see Corollary \ref{cor:strongly_efficient} in Section \ref{s:algorithm}}.
As an aside we remark that we hypothesize that this decay is actually of order $n^{-1/2}$. This is supported by numerical experiments; see Figure $\ref{fig:relative_bias_M=1}$ where plots of $\beta_n(b)$ versus $n$ are shown for the \red{threshold-dependent} grid (\ref{THE_grid1}). The step we expect to be `loose', \red{in obtaining the bound of Theorem  \ref{THEorem},} is the one corresponding to Lemma \ref{lemma2_me}.
We conjecture that Lemma \ref{lemma2_me} is valid with
\[N_n:=\frac{t_n}{\max_{k=1,\ldots,n}(t_k - t_{k-1})}.\]
(i.e.,  without the square root), which suffices to yield the $n^{-1/2}$ decay of $\beta_n(b)$. }
\label{remark1}
\end{remark}
}
\begin{figure}[h]
 \begin{adjustwidth}{0cm}{}
        \centering
       \includegraphics[scale=.5]{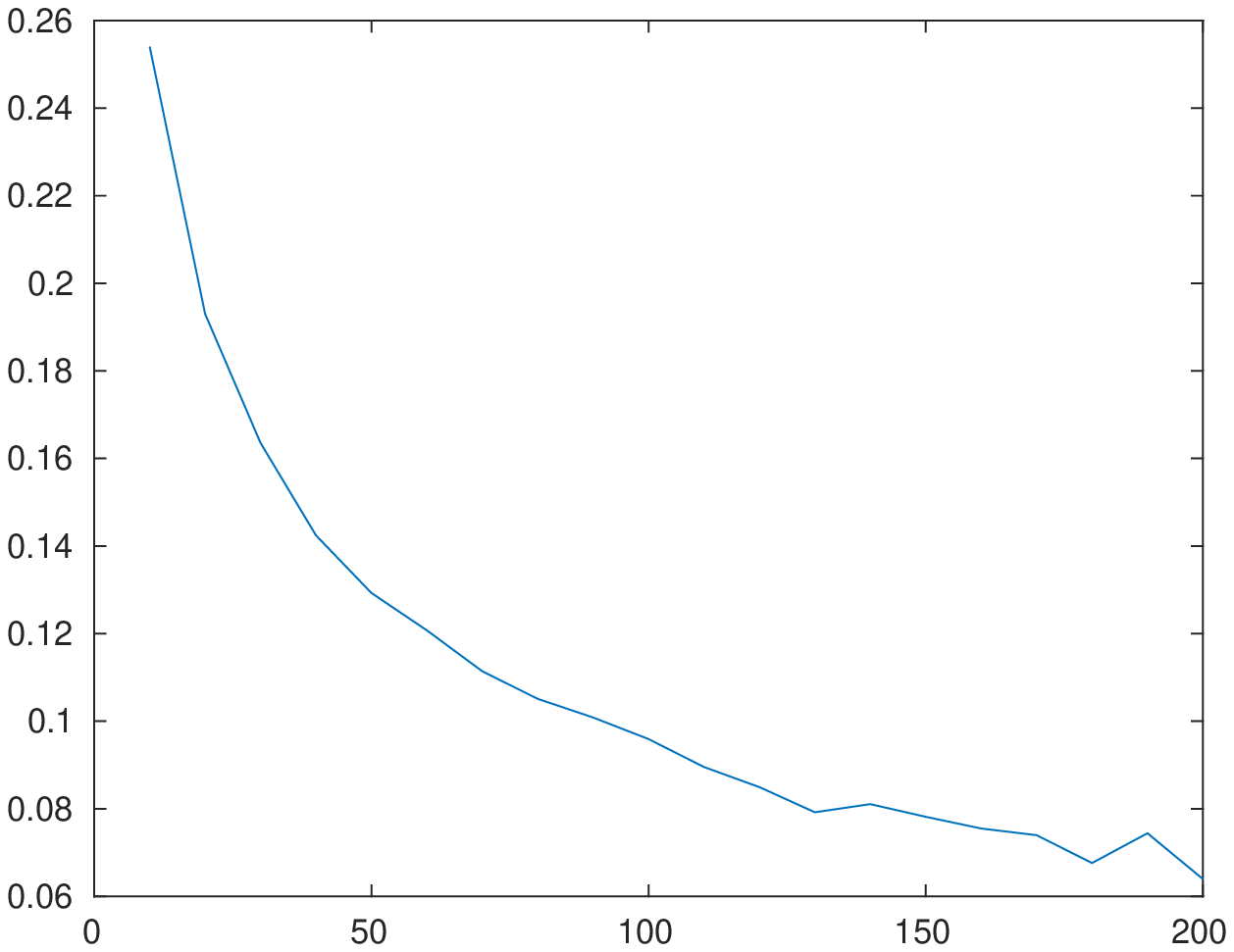}
       \includegraphics[scale=.5]{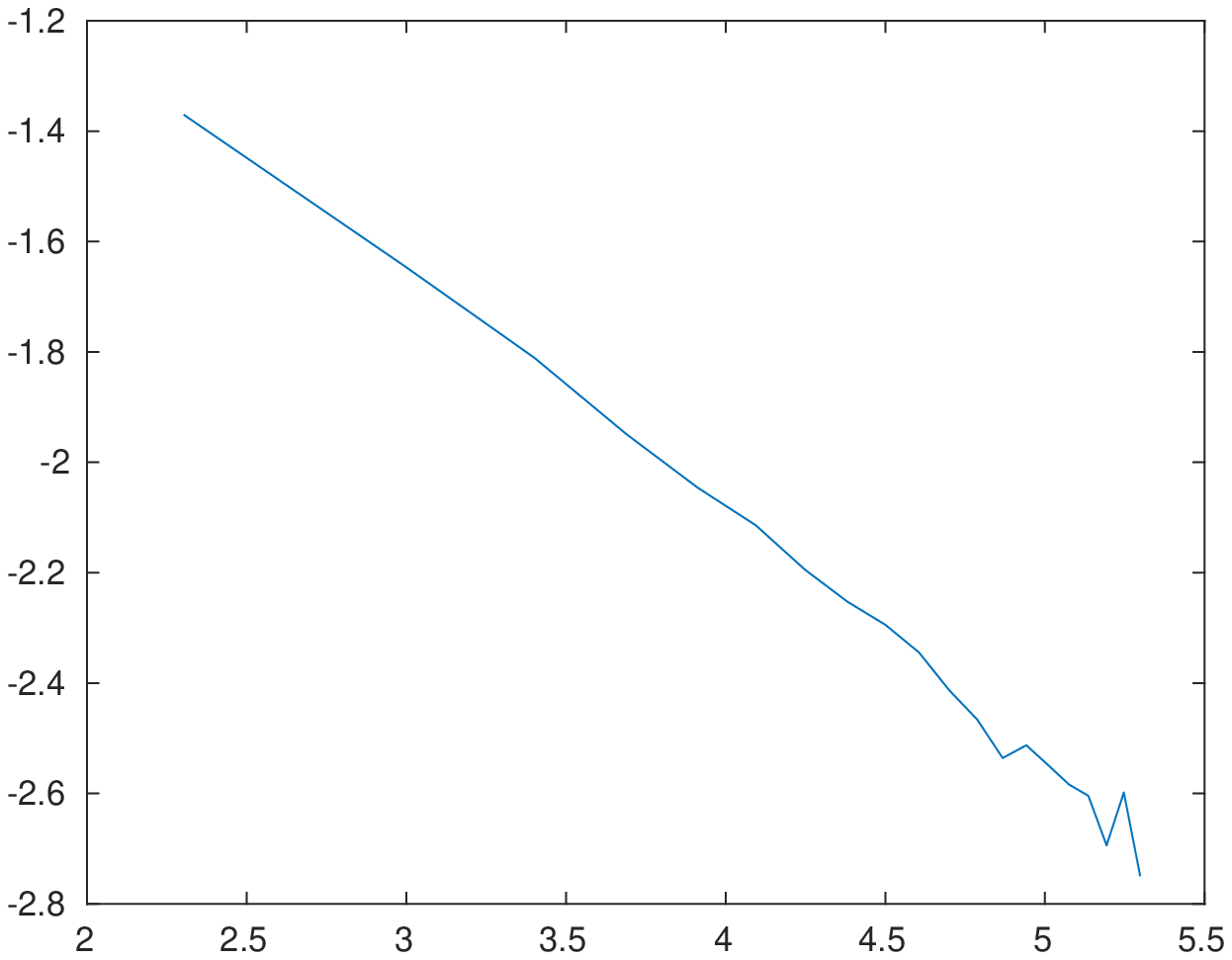}
      \caption{Relative bias $\beta_n(b)$ versus grid size $n$ for the \red{threshold-dependent} grid (\ref{THE_grid1}). The threshold is fixed at $b=3$. The right panel shows a loglog plot, left panel a linear plot. The results suggest that $\beta_n(b)$ decays proportionally to $n^{-1/2}$ rather than $n^{-1/4}$ \red{(see also Remark \ref{remark1})}.
      }\label{fig:relative_bias_M=1}
       \end{adjustwidth}
\end{figure}


\section{Numerical algorithm for estimation of $w(b)$}\label{s:algorithm}

As mentioned in the introduction, the family of \red{threshold-dependent} grids (\ref{THE_grid1}) can be used to construct a strongly efficient algorithm for estimation of $w(b)$, see Corollary \ref{cor:strongly_efficient} below. In this paper, by `strongly efficient' we mean that for any given accuracy $\varepsilon>0$ and confidence level $\alpha>0$ the computational time of an estimator $\widehat{w}(b)$ for $w(b)$ that satisfies
\begin{align}
\PPP\left( \left| \frac{\widehat{w}(b)}{w(b) }-1 \right| > \varepsilon \right) < \alpha
\label{confidence}
\end{align}
is bounded independently of the threshold $b$.\\

In all numerical experiments throughout this paper we used an algorithm developed by \cite{adler2012efficient}, see Algorithm \ref{alg:blanchet} below. Although it is applicable for estimation of quantities such as $\PPP(\max_{i\in\{1,\ldots,n\}} X_i > b)$, where $X\in\R^n$ is normally distributed with an arbitrary positive-definite covariance matrix, we present their algorithm for the specific case of Brownian Motion, as considered in this paper.
\begin{algorithm}[\cite{adler2012efficient}]\label{alg:blanchet}
Choose a threshold $b$ and a finite grid $T = \{t_1, \ldots, t_n\} \in [0,1]$. The estimator $\widehat{w}_T(b)$, computed according to the following algorithm, is an unbiased estimator of $w_T(b)$.
\begin{enumerate}
\item Generate a random time $\tau$ on the grid, i.e. $\tau \in T$, according to the law \[\PPP(\tau = t_k) = \frac{\PPP(B_{t_k}>b)}{\sum_{j=1}^n \PPP(B_{t_j} > b)}.\]
\item Generate $B_\tau$ under the condition $B_\tau > b$.
\item Generate a discrete path of the Brownian Motion $(B_{t_1}, \ldots, B_{t_n})$ conditioned on the pair $(\tau, B_\tau)$ generated in the previous steps.
\item Compute \[\widehat{w}_T(b) := \frac{\sum_{j=1}^n \PPP(B_{t_j} > b)} {\sum_{j=1}^n \ind(B_{t_j} > b)}.\]
\end{enumerate}
\end{algorithm}
\cite{adler2012efficient} prove that the Algorithm \ref{alg:blanchet} gives an \emph{unbiased} estimator of $w_T(b)$ (not of $w(b)$) and that for a fixed $T$ (independent of $b$), the relative variance $\Var(\widehat{w}_T(b))/w_T^2(b) \to 0$, as $b\to\infty$. The authors also propose an estimator for $w(b)$, which relies on a \textit{random discretization}. However, with growing $b$, one needs increasingly many random grid-points in order to control the relative bias, therefore the continuous-time algorithm \textit{is not} strongly efficient. In order to reduce the sampling error one generates multiple replicas of the estimator and takes their average. Since every replica is based on a different grid, one must repeatedly calculate the Cholesky decomposition (whose computational time is cubic in the number of grid-points) in order to sample discrete Gaussian paths in Step 3 of Algorithm \ref{alg:blanchet}. Choosing a predefined grid speeds up this computation, as in that case the Cholesky decomposition has to be performed \emph{only once}, making its computational cost negligible.\\

Combining the \red{threshold-dependent} grids as proposed in Section \ref{s:BM_bb} with Algorithm \ref{alg:blanchet} yields a strongly efficient estimator for $w(b)$ which is given in the corollary below.

\begin{corollary}[Strongly efficient algorithm for the estimation of $w(b)$]\label{cor:strongly_efficient} Fix an accuracy $\varepsilon>0$ and a confidence level $\alpha>0$. Choose a grid $T := T_n(b)$ from the family of grids defined in $(\ref{THE_grid1})$ such that $\beta_T(b) := \beta_n(b) < \varepsilon$ \red{for all $b>0$} (this is possible due to the result in Theorem \ref{THEorem}). Let $\widehat{w}^{(1)}_{T}(b), \ldots \widehat{w}^{(N)}_{T}(b)$ be i.i.d copies of the estimator from Algorithm \ref{alg:blanchet}, with \[N \geq \frac{n^2}{\alpha (\varepsilon-\beta_T(b))^2}.\] Then \[\widehat{w}(b) := \frac{1}{N}\sum_{i=1}^N \widehat{w}^{(i)}_T(b)\] satisfies
\begin{align}\label{to_prove:strong_eff}
\PPP\left( \left| \frac{\widehat{w}(b)}{w(b) }-1 \right| > \varepsilon \right) < \alpha,
\end{align}
and the computational effort to simulate $\widehat{w}(b)$ is bounded independently of $b$.
\end{corollary}

\begin{proof}
First notice that since $\beta_T(b)$ is uniformly bounded in $b$ (see Theorem \ref{THEorem}), so that $N$ is fixed independently of $b$, it follows that $\widehat{w}(b)$ can be computed in bounded time, independently of $b$. It remains to prove that $\widehat{w}(b)$ satisfies the strong efficiency property (\ref{to_prove:strong_eff}). Note that $\widehat{w}(b)$ is an unbiased estimator of $w_T(b)$, not of $w(b)$. The relative variance of $\widehat{w}(b)$ with respect to $w_T(b)$ can be bounded independently of $b$ for an arbitrary choice of the grid in terms of the grid size $n$,
\begin{align*}
\frac{\Var(\widehat{w}_T(b))}{(w_T(b))^2} \leq \frac{\Exp(\widehat{w}_T(b))^2}{(w_T(b))^2} \leq \left( \frac{\sum_{j=1}^n \PPP(B_{t_j} > b)}{\max_{j\in\{1,\ldots,n\}} \PPP(B_{t_j} > b)}\right)^2 \leq n^2.
\end{align*}
Due to Chebyshev's inequality,
\begin{align*}
\PPP\left( \left| \frac{\widehat{w}(b)}{w(b)} - 1 \right| > \varepsilon \right) & = \PPP\left( \left| \frac{\widehat{w}(b) - w_T(b)}{w(b)} + \frac{w_T(b) - w(b)}{w(b)} \right| > \varepsilon \right) \leq \PPP\left( \left| \frac{\widehat{w}(b) - w_T(b)}{w(b)} \right| > \varepsilon - \beta_T(b) \right) \\
& \leq \frac{\Var(\widehat{w}(b))}{(\varepsilon-\beta_T(b))^2 (w(b))^2}  = \frac{1}{N} \cdot \frac{\Var(\widehat{w}_T(b))}{(\varepsilon-\beta_T(b))^2 (w(b))^2} \\
& \leq \frac{1}{N} \cdot \frac{n^2}{(\varepsilon-\beta_T(b))^2} \leq \alpha.
\end{align*}
This concludes the proof.
\end{proof}

We conclude this section by a remark on the simulation of the conditioned Brownian Motion in Step 3 of Algorithm \ref{alg:blanchet}. The na\"{\i}ve method would be to construct the covariance matrix of the conditioned process, calculate the Cholesky decomposition of that matrix (cubic in the number of grid points) and then simulate the process in a standard manner. Notice that this step must be repeated for every replica $\widehat{w}^{(i)}_T(b)$ and thus its computational cost scales with the number of samples. The following algorithm, which can be found e.g. in \cite{doucet2010note}, requires only a single calculation of the Cholesky decomposition for all replicas.

\begin{algorithm}[\cite{doucet2010note}]\label{alg:doucet}
Let $X = (X_1, X_2)^T \in \R^n$, where $X_1 \in \R^{n-1}$ and $X_2 \in \R$, be normally distributed with mean $\mu$ and covariance matrix $\Sigma$,
\begin{align*}
& \mu = \begin{pmatrix}
\mu_1\\
\mu_2
\end{pmatrix}, \text{ where } \mu_1\in\R^{n-1} \text{ and } \mu_2\in\R \, , \\
& \Sigma = \begin{pmatrix}
\Sigma_{11} & \Sigma_{12} \\
\Sigma_{12}^T & \Sigma_{22}
\end{pmatrix}, \text{ where } \Sigma_{11}\in\R^{(n-1)\times (n-1)}, \Sigma_{12}\in\R^{n-1} \text{ and } \Sigma_{22}\in\R.
\end{align*}
The following algorithm generates a sample $\overline{X} \sim (X_1 | X_2 = x_2)$:
\begin{enumerate}
\item Sample $Z = (X_1, X_2)^T \sim N(\mu,\Sigma)$
\item Compute $\overline{X} = X_1 + \Sigma_{12}\Sigma_{22}^{-1}(x_2 - X_2)$.
\end{enumerate}
\end{algorithm}

Note that the computational effort to produce the conditioned Gaussian random variable $\overline{X}$ in Step 2 of Algorithm \ref{alg:doucet} is linear in the dimension $n$. Thus, this algorithm significantly reduces the computation time of Step 3 of Algorithm \ref{alg:blanchet} when that step is repeated for each replica.


\red{
\section{Efficient grids for a broad class of  stochastic processes}\label{s:application}

In this section we discuss how the idea of threshold-dependent grids can be applied to  stochastic processes other than Brownian Motion.
We let $(X_t)_{t\in[0,1]}$ be a real-valued stochastic process and $ t^*(b) := \argmax_{t\in[0,1]} \PPP(X_t > b)$. 
For simplicity we here assume that $t\mapsto\PPP(X_t>b)$ is  {continuous and  strictly increasing} so that $t^*(b)=1$ (but situations in which $t^*(b)\in(0,1)$ can be dealt \red{with similarly, see also the discussion in Section \ref{s:discussion}}). 

As argued in the previous sections, it is efficient to let the position of the grid points  depend on $b$. We constructed  for Brownian Motion a grid by finding $T(b)=\{t_1(b),\ldots,t_n(b)\}$ such that
\begin{equation}\label{G1}{\mathbb P}\big( \tau_b\in({0},t_k(b)]\,|\, \tau_b\in(0,1]\big) =\frac{k}{n};\end{equation}
\red{cf.\ (\ref{THE_grid3}).} An inherent problem is that the class of processes for
which the distribution of $\tau_b$ is known is very limited, so that the approach does not seem to be useful for relevant stochastic processes other than Brownian motion. We saw, however, that for Brownian Motion the $t_k(b)$ satisfying (\ref{G1}) also solve 
\begin{equation}\label{G2}\frac{{\mathbb P}(X_{t_k(b)}>b)}{{\mathbb P}(X_1>b)}=\frac{k}{n};\end{equation}
\red{cf.\ (\ref{THE_grid4})}. The idea now is to use the level-dependent (or: `equiprobable') grid (\ref{G2}) for general real-valued processes. The \red{major} advantage of the grid (\ref{G2}) is that to calculate the position of the grid points $t_k$ the sole prerequisite is that the process' {\it marginals} are known (rather than the distribution of $\tau_b$). In addition, even if the marginal distributions of $X_t$ are not available, but  the {\it asymptotics} of $\mathbb{P}(X_t>b)$ (as $b\to\infty$) are, then  a good approximation of this grid can be found. 
(In the sequel we write, for brevity, $T=\{t_1,\ldots,t_n\}$ instead of $T(b)=\{t_1(b),\ldots,t_n(b)\}$) and $t^*$ instead of $t^*(b).$)

\vb

We now provide the rationale behind the grid (\ref{G2}).  Let $T$ be a grid such that $t^*\in T$. Evidently, by the union bound, \[\PPP(X_{t^*} > b) \leq w_T(b) \leq \sum_{t\in T} \PPP(X_t>b)\]
Now notice that if the grid $T$ is such that for $t\in T\setminus \{t^*\}$
\begin{align}\label{def:as_no_contr}
\PPP(X_t>b) = o\big(\PPP(X_{t^*}>b)\big), \ \ \text{ as } b\to\infty
\end{align}
then it does not make sense to include the point $t$ \red{for large $b$}. Property (\ref{def:as_no_contr}) clearly compromises the performance of equidistant grids as $b\to\infty.$
Considering however the grid points $t_k$ of the \red{threshold}-dependent grid, as defined by (\ref{G2}), these will by design not experience (\ref{def:as_no_contr}).


\vb

To assess the performance of the above threshold-dependent grid (\ref{G2}), we introduce a measure of performance closely related to the relative bias. Note that 
when no formulas for $w(b)$ are available, nor it is known how to reliably approximate $w(b)$, we cannot determine the exact value of the relative bias.
We now make the following two observations. (1) As $w_T(b)<w(b)$ for any choice of $T$, the larger $w_T(b)$ is, the better; if $w_{T_1}(b)>w_{T_2}(b)$ for grids $T_1,T_2$, then also $\beta_{T_1}(b)<\beta_{T_2}(b)$. (2) The crude lower bound $w(b)\geq\PPP(X_{t^*}>b)$  provides us with a useful benchmark. Combining these two thoughts motivates  the following performance measure  of a grid $T$:
\begin{align*}
\gamma_T(b) := \frac{w_T(b)}{\PPP(X_{t^*}>b)}
\end{align*}
Notice that for any $T$ such that $t^*\in T$ we have \[\gamma_T(b) \in\bigg[1,\frac{w(b)}{\PPP(X_{t^*}>b)}\bigg].\] What is more, for any two grids $T_1,T_2$ we have $\gamma_{T_1}(b) \geq \gamma_{T_2}(b)$ if and only if $\beta_{T_1}(b) \leq \beta_{T_2}(b)$; this means that the bigger the $\gamma_T(b)$ is, the better. \red{As our main aim is to efficiently approximate $w(b)$ using discrete-time approximations $w_T(b)$, we see that if $\gamma_T(b)\approx 1$ then there is little gain from using $w_T(b)$ over a deterministic estimator $\PPP(X_{t^*}>b)$.}

\vb

In a series of examples we compare $\gamma_T(b)$ induced by (i)  the threshold-dependent (equiprobable) grid and (ii) the equidistant grid of the same size; we consistently use  $n=100$ grid points. In all cases $t\mapsto\PPP(X_t>b)$ is a continuous, strictly increasing function (so that $t^*=1$). The most important conclusion is that the experiments below uniformly indicate that the equiprobable grid outperforms the equidistant one, not only in the asymptotic regime, as threshold $b$ grows large, but already for moderate values of $b$. This shows how the ideas the we developed earlier this paper, that have provable optimality properties for Brownian motion, lead to an efficient estimation procedure for a much broader class of stochastic processes. 
In all examples, we observe that $\gamma_T(b)$ induced by the equidistant grid  converges to $1$, \red{thus the corresponding $w_T(b)$ is asymptotically equivalent to $\PPP(X_{t^*}>b)$, as $b\to\infty$.}

\begin{example}[Brownian Motion with jumps] {\em 
Let $(X_t)_{t\in[0,1]}$ be a Brownian Motion with jumps, i.e.
\begin{align}\label{def:bm_with_jumps}
X_t := B_t + N_t,
\end{align}
where $B_t$ is a standard Brownian Motion and $N_t$ is a standard Poisson process with intensity $\lambda=1$.

Even though there are no closed-form expressions for $w(b)$, it is still possible to generate exact samples from $\sup_{t\in[0,1]} X_t$ (see \cite[Section 10.1]{dkebicki2015queues}). We can use this to construct an unbiased estimator of $w(b)$ and thus can estimate the relative bias of the tested grids. The results in Figure \ref{fig:bm_poiss_bias} show the substantial gain achieved by the level-dependent grid. The graphs look  similar to those of Brownian Motion, which is indicative of  the threshold-dependent grid having a uniformly bounded relative bias.}
\begin{figure}[h]
  \centering
  \begin{subfigure}{.5\textwidth}
    \centering
    \includegraphics[width=1.00\linewidth]{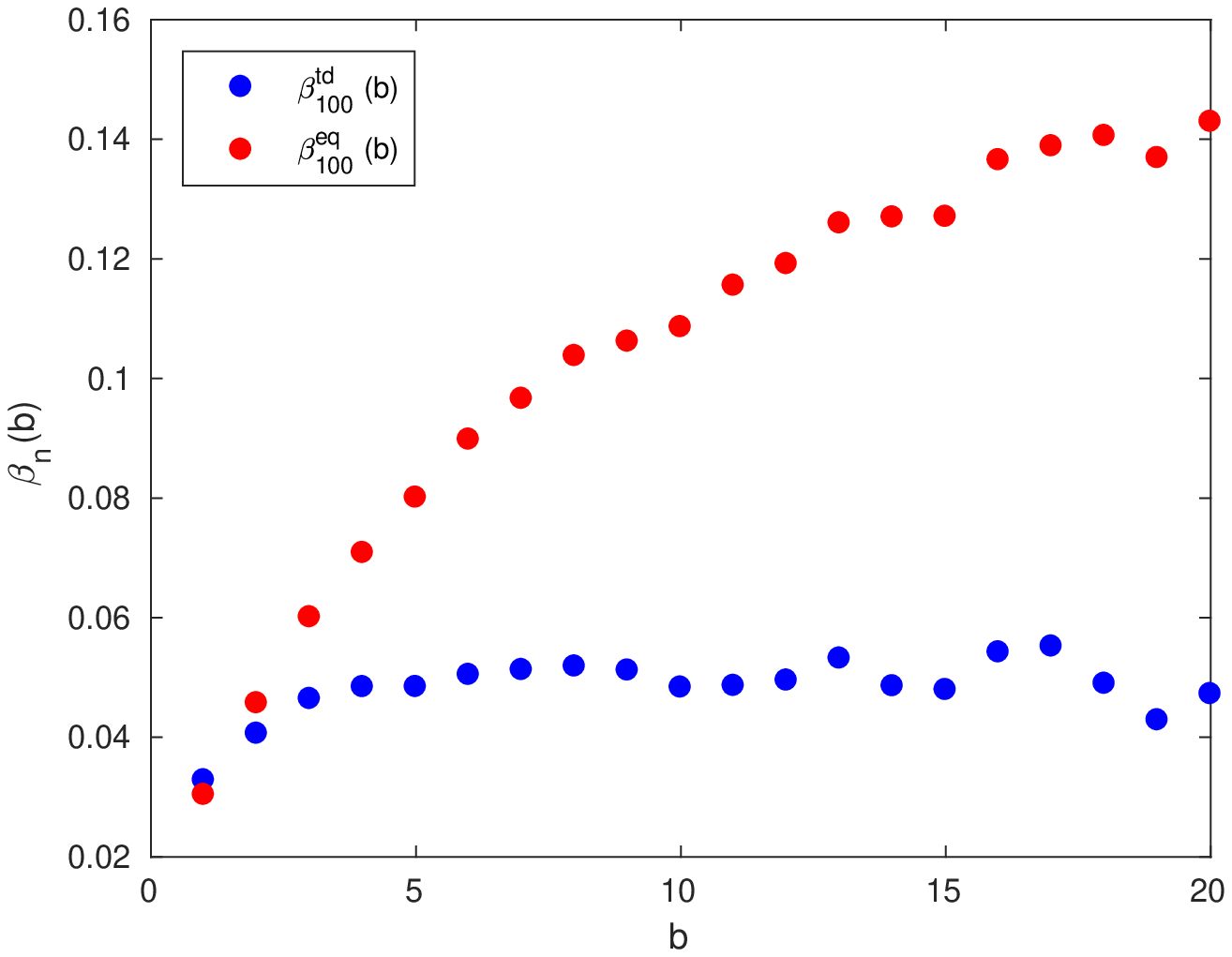}
  \end{subfigure}\hfill
  \begin{subfigure}{.5\textwidth}
    \centering
    \includegraphics[width=1.00\linewidth]{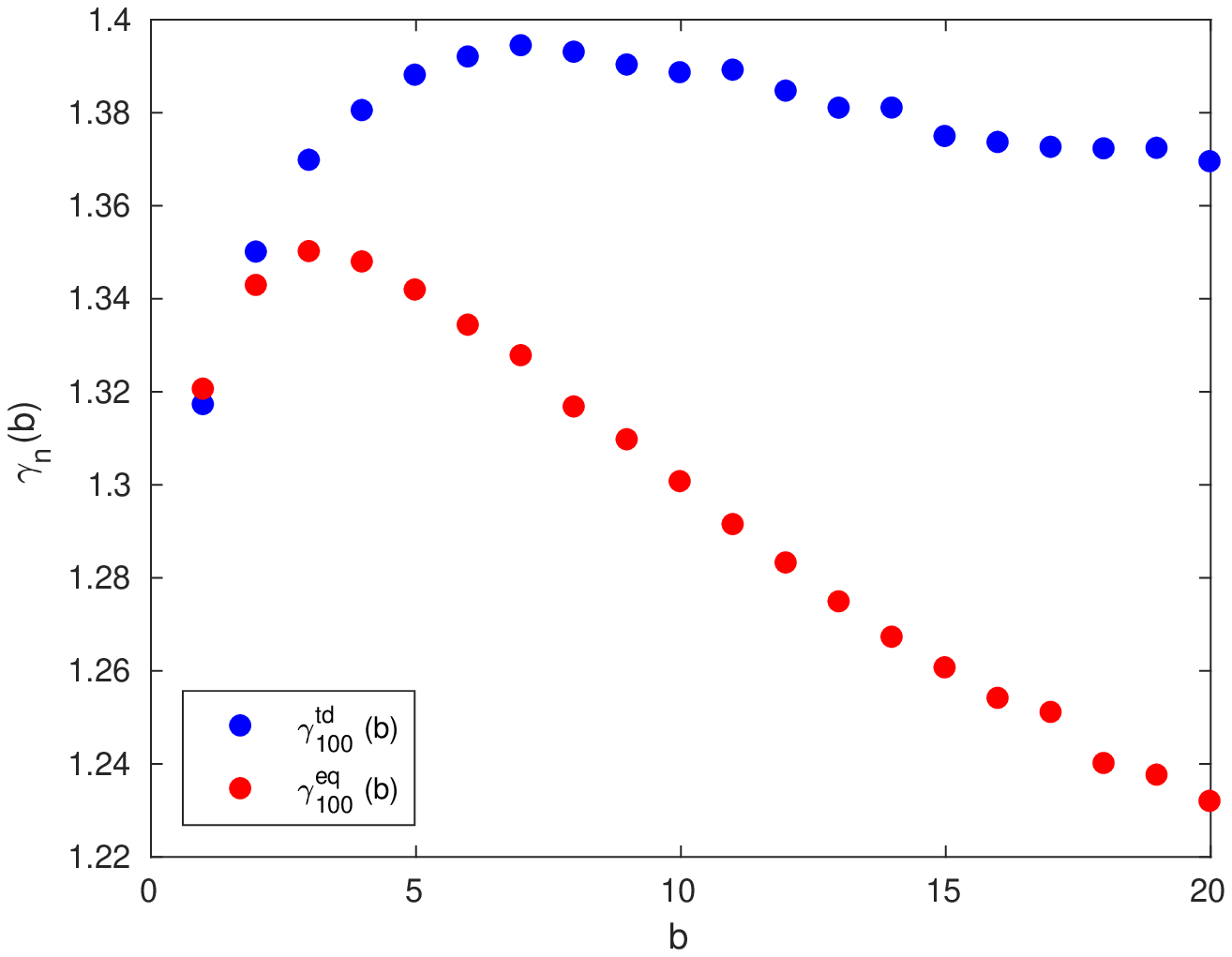}
  \end{subfigure}
    \caption{Brownian Motion with jumps. Plots of $\beta_n(b)$ (left) and $\gamma_n(b)$ (right) as a function of the threshold $b$ for threshold-dependent and equidistant grids of size $n=100$.
      }\label{fig:bm_poiss_bias}
\end{figure}

\end{example}

\begin{example}[Ornstein-Uhlenbeck Process]\label{ex:num_exp2}{\em 
Let $(X_t)_{t\in[0,1]}$ be an Ornstein-Uhlenbeck process, i.e., a strong solution to the following SDE: with $X_0=0$,
\begin{align*}
dX_t &\,= -X_t\,dt + dW_t.
\end{align*}
Then $(X_t)_{t\in[0,1]}$ is a zero-mean Markovian Gaussian process with covariance function \[c(s,t) := {\mathbb C}{\rm ov}(X_s,X_t) = \frac{1}{2} \left( e^{-|t-s|} - e^{-(t+s)} \right).\] \red{The exact value of $w(b)$ is known only in terms of special functions, see \cite{alili2005representations} and it is not straightforwardly evaluated. However, the exact asymptotics of $w(b)$, as $b$ grows large, {\it are} known:}
\begin{align*}
w(b) = C\, \PPP(X_1>b) (1+o(1)), \text{ as } b\to\infty
\end{align*}
where $C$ is a positive constant independent of $b$, \red{see e.g.\ \cite[Theorem 5.1]{dkebicki2003exact} or the original theorem by \cite{piterbarg1978asymptotic}}; this explains why for the level-dependent grid $\gamma_n(b)$ goes to a constant in Figure \ref{fig:ou_bias}. Again the equidistant grid is significantly outperformed by the threshold-dependent grid.}

\begin{figure}[h]
  \centering
  \begin{subfigure}{.75\textwidth}
    \centering
    \includegraphics[height=0.7\linewidth, width=1\linewidth]{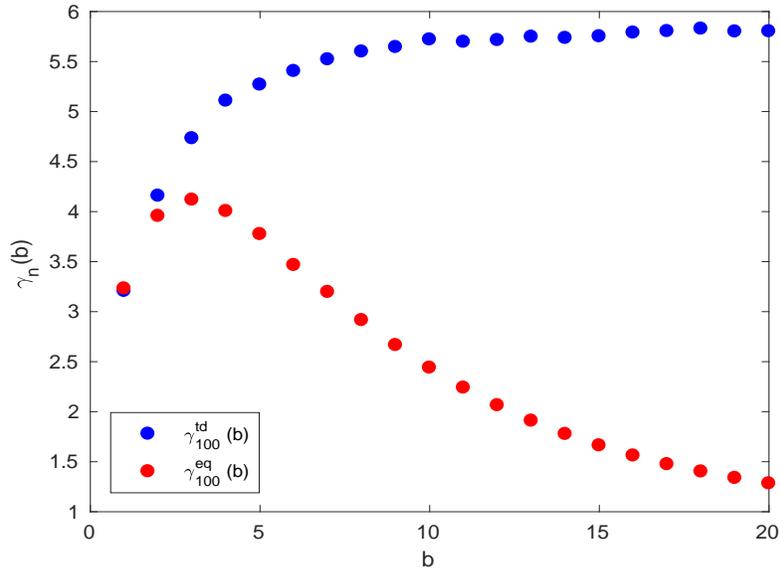}
  \end{subfigure}
    \caption{Ornstein-Uhlenbeck process. Plot  of $\gamma_n(b)$ as a function of the threshold $b$ for threshold-dependent and equidistant grids of size $n=100$. Notice that with growing $b$ the equidistant estimator converges to $\PPP(X_1>b)$. 
      }\label{fig:ou_bias}
\end{figure}

\end{example}

\begin{example}[Fractional Brownian Motion]\label{ex:num_exp3}{\em 
Let $(X_t)_{t\in[0,1]}$ be a fractional Brownian Motion (fBM) with a Hurst parameter $H\in(0,1)$, that is a zero-mean Gaussian process with the covariance function
\begin{align*}
C_H(s,t) := {\mathbb C}{\rm ov}(X_s,X_t) = \frac{1}{2}\left(s^{2H} + t^{2H} - |t-s|^{2H}\right).
\end{align*}
Observe that fBM with Hurst parameter $H=1/2$ is a standard Brownian Motion. For any $H$ we have $C_H(t,t) = t^{2H}$ (strictly increasing variance in time) and thus $t^* = 1$.\\

\red{The exact value of the probability $w(b)$ for $H\neq 1/2$ remains unknown. However, like in Example \ref{ex:num_exp2}, the exact asymptotics of $w(b)$ are known:}
\begin{align*}
w(b) = \begin{cases}
C_H b^{1/H-2}\PPP(X_1>b)(1+o(1)), & \text{ for } H\in(0,\frac{1}{2}) \\
\PPP(X_1>b)(1+o(1)), & \text{ for } H\in(\frac{1}{2},1)
\end{cases}
\end{align*}
where $C_H$ is a constant only depending on $H$; \red{we again refer to \cite[Theorem 5.1]{dkebicki2003exact} or the original theorem by \cite{piterbarg1978asymptotic}}. We apply threshold-dependent grids in these two different asymptotic regimes for $H=0.4$ and $H=0.6$, see the results in Figure~\ref{fig:fbm_bias}. Again the threshold-dependent grid performs considerably better. In case $H=0.4$ the above asymptotic result explains why for the level-dependent grid $\gamma_n(b)$ keeps increasing ($w(b)/{\mathbb P}(X_1>b)$ behaves as the increasing function $b^{1/H-2}$). In case $H=0.6$, again using the asymptotic result, 
\red{$\gamma_n(b) \to 1$ as $b$ grows large, both for the equidistant grid and for the threshold-dependent grid (equivalently, the relative bias vanishes for both as $b \to \infty$).  Note however that with the threshold-dependent grid, $\gamma_n(b)$ tends to 1 slower than with the equidistant grid, as can be seen in Figure \ref{fig:fbm_bias} \red{(right panel)}, showing the more favorable performance of the threshold-dependent grid.}
}

\begin{figure}[h]
  \centering
  \begin{subfigure}{.5\textwidth}
    \centering
    \includegraphics[width=1.00\linewidth]{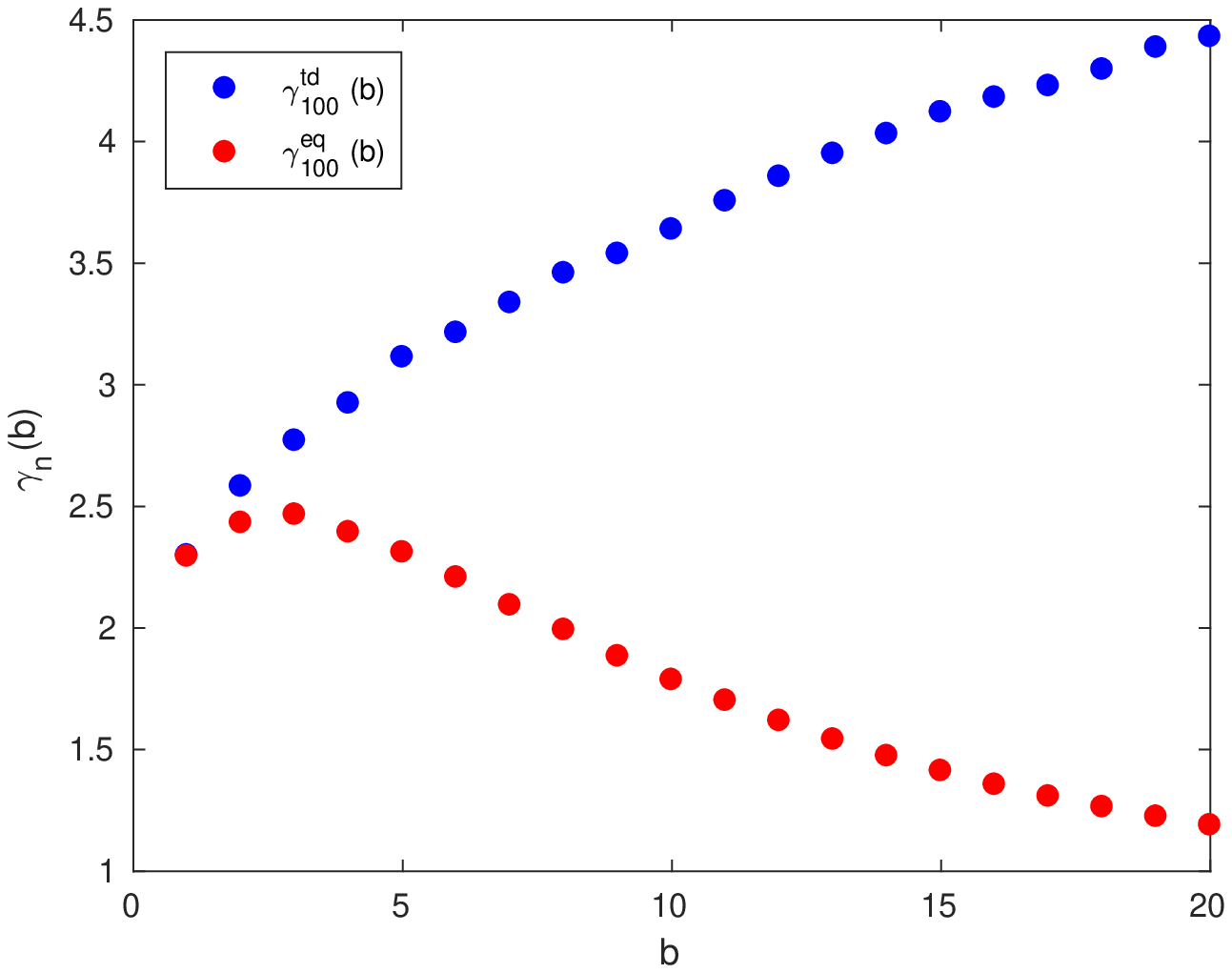}
  \end{subfigure}\hfill
  \begin{subfigure}{.5\textwidth}
    \centering
    \includegraphics[width=1.00\linewidth]{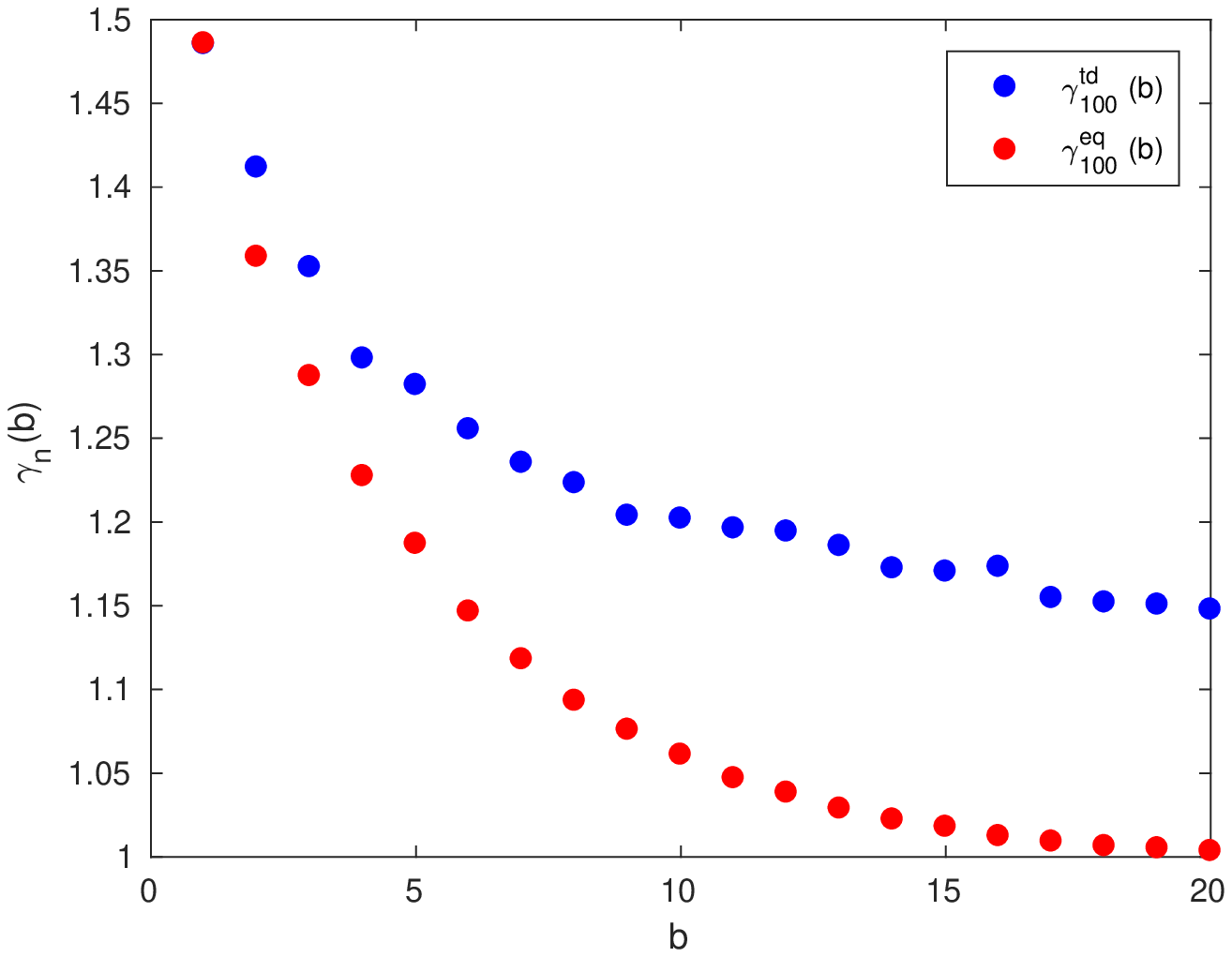}
  \end{subfigure}
    \caption{fBm with Hurst parameter $H=0.4$ (left) and $H=0.6$ (right). Plot of $\gamma_n(b)$ as a function of the threshold $b$ for threshold-dependent and equidistant grids of size $n=100$.
      }\label{fig:fbm_bias}
\end{figure}

\end{example}

}


\section{Concluding remarks and discussion}\label{s:discussion}

In this paper we have demonstrated that the errors due to time discretization when estimating threshold-crossing probabilities $w(b)$ can be significantly reduced  by using other grids than the commonly used equidistant grid. We have analyzed this in considerable detail for the case of standard Brownian Motion. In particular, we have shown that in order to control the error as $b$ grows large, it suffices to properly \emph{shift} the grid points instead of refining the grid with more and more points. At the same time, controlling the error using equidistant grids requires \emph{quadratic} growth of the number of grid points, as $b$ grows large.

\vb

Numerical estimation is evidently not needed for Brownian Motion due to the availability of analytical results. \red{Our paper however indicates that the underlying ideas can be used to construct  efficient grids for a broad class of stochastic processes (notably, \red{L\'evy processes} and Gaussian processes, such as fractional Brownian Motion). The results presented in this paper are intended to develop valuable insight and useful heuristics for tackling the estimation of tail probabilities of these more general classes of processes. \red{We have demonstrated such heuristics for several processes in Section \ref{s:application}. There,} we presented a procedure, that is empirically shown to work well for stochastic process $(X_t)_{t\in[0,1]}$ of which the marginal distributions are known: 
\begin{itemize}
\item[(i)] Identify \[t^*(b) := \argmax_{t\in[0,1]} \PPP(X_t>b);\] in case $(X_t)_{t\in[0,1]}$ is a zero-mean Gaussian process, $t^*$ is a point of maximal variance, i.e., $\argmax_{t\in[0,1]} \Var\, X_t$. As argued, for many key models we have that $t^*=1.$
\item[(ii)] Construct a grid $T = \{t_1,\ldots,t_n\}$ clustered around it, such that $t_k$ solves (\ref{G2}), for $k\in\{1,\ldots,n\}$.
\end{itemize} 
As we pointed out, even if the marginal distribution of $X_t$ is not available but only the corresponding asymptotics, as $b\to\infty$, this procedure can be applied.} It is also noted that it is straightforward to compare two different grids:  the larger the value of $w_T(b)$, the closer it is to the target quantity  $w(b)$.

\vb

A natural question that arises in relation to Theorem \ref{THEorem} is whether we can find a grid that is even better than the one defined in (\ref{THE_grid1}). Constructing an \textit{optimal n-grid} $T^*_n(b)$, i.e. a grid of size $n$ that minimizes the relative bias for a given $b$, remains elusive. However we have been able to find an explicit formula for an optimal 2-grid, namely $T^*_2(b) = \{t^*_1(b), t^*_2(b)\}$, with
\begin{align*}
t^*_1(b) = \frac{\pi b^2}{4} \left( \sqrt{1+\frac{8}{\pi b^2}} - 1\right), \ \ \ \ \text{and} \ \ \ \  t^*_2(b) = 1
\end{align*}
where $\limb \beta_{T^*_2(b)}(b) = 1 - \frac{1}{2} \Phi(\sqrt{2/\pi}) - \frac{1}{4}e^{-1/\pi} \approx 0.4244$. For comparison, the \red{threshold-dependent} grid defined in (\ref{THE_grid1}) yields $\limb \beta_2(b) = \frac{3}{8} + \frac{1}{2} \Phi(-\sqrt{2\log2}) \approx 0.4348$, hence the grid (\ref{THE_grid1}) is not minimizing the bias (although the difference with the optimal 2-grid is small). Additionally, we were able to prove that for an optimal n-grid, $T^*_n(b) = \{t_1^*(b),\ldots,t_n^*(b)\}$, the limits $\limb b^2(1-t^*_k(b))$ must exist, and are all finite and pairwise distinct. As a result we were able to numerically calculate the limit $\limb \beta_{T^*_3(b)} \approx 0.3796$. Finding optimal grids for larger $n$ remains an open problem. We note, however, that with the \red{threshold-dependent} grid we can bound the relative bias uniformly in $b$ (see Theorem \ref{THEorem}) and in this sense the grid (\ref{THE_grid1}) is already (asymptotically) optimal.


\section{Proofs of Lemmas \ref{lem:BM_LB}, \ref{equi_grid_increasing_b}, \ref{lemma2_me}, \ref{lemming} and Proposition \ref{prop:eq}}\label{s:proofs}
\begin{proof}[Proof of Proposition \ref{prop:eq}]
In part (a) of Theorem \ref{thm_eq} it has been proven already  that $\beta_n(b) \leq C_0 bn^{-1/2}$. Thus, when $b$ is fixed it is straightforward that the upper bound in the assertion of the theorem holds.\\
The lower bound developed in Lemma \ref{lem:BM_LB} reads $\beta_n(b) \geq \frac{1}{2} \sum_{j=1}^{n-1} a_{j+1}\cdot w_j(b) + \frac{1}{2}w_n(b)$. Since we have $a_j < a_{j+1}$ for the equidistant grid and all $a_j$ and $w_j$ are non-negative, we may use the weaker inequality \[\beta_n(b) \geq \frac{1}{2}\sum_{j=2}^n a_j\cdot w_j(b).\]
In the following we use Lemma \ref{lem_feller} for a lower bound on terms $a_j$ and Result \ref{appendix:results}.\ref{mvt_taub} for a lower bound on  $w_j$.
\begin{align}
\nonumber \sum_{j=2}^{n} a_{j}\cdot w_j(b) & \geq \frac{b\,(3b + \sqrt{b^2 +8})}{8} \, \sum_{j=2}^{n} C_1^*(n-j+1)^{-1/2}\frac{\sqrt{n}}{j^{3/2}} \, e^{-\frac{b^2}{2} \, \left(\frac{n}{j-1}-1\right)} \\
\nonumber & \geq C \, n^{-1/2} \, \sum_{j=2}^{n} \frac{1}{n} \, \left( \frac{b}{\sqrt{1-\frac{j-1}{n}}} \, \left(\frac{j-1}{n}\right)^{-3/2} e^{-\frac{b^2}{2} \, \left( \frac{n}{j-1} - 1\right)} \right) \\
\label{riemann_sum_prop} & \geq C \, n^{-1/2} \, \int_0^1 \frac{b}{\sqrt{1-x}} \, x^{-3/2} \, e^{-\frac{b^2}{2}\, (1/x - 1)}\,dx \\
\nonumber & \geq C \, n^{-1/2},
\end{align}
where $C$ is a positive constant independent of $n$ (but dependent on $b$) that may vary from line to line. To arrive at (\ref{riemann_sum_prop}) we use the convergence of the Riemann sum, noting that $b$ is fixed and that the function \[f(x) := \frac{b}{\sqrt{1-x}} \, x^{-3/2} \, e^{-\frac{b^2}{2}\, (1/x - 1)}\] is integrable on $(0,1)$. This concludes the proof.
\end{proof}

\begin{proof}[Proof of Lemma \ref{lem:BM_LB}]
Notice that the events $\{\sup_{t\in [0,1]} B_t > b\}$ and $\{\tau_b\in(0,1]\}$ are equivalent. We thus find
\begin{align*}
w(b)\,\beta_T(b) & = \PPP\Big(\sup_{t\in T} B_t < b, \sup_{t\in [0,1]} B_t > b \Big) = \PPP\Big(\sup_{t\in T} B_t < b, \tau_b\in[0,1] \Big) \\
& = \sum_{j=1}^n \PPP\Big(\sup_{t\in \{t_j, \ldots, t_{n}\}} B_t < b, \tau_b\in(t_{j-1},t_j] \Big) \\
& = \sum_{j=1}^n \int_{t_{j-1}}^{t_j} \PPP\Big( \sup_{t\in \{t_j, \ldots, t_{n}\}} B_t < b \mid B_s = b \Big)\,\PPP(\tau_b\in ds) \\
& = \sum_{j=1}^n \int_{t_{j-1}}^{t_j} \PPP\Big( B_{t_j-s} < 0, \ldots, B_{t_n-s} < 0 \Big)\,\PPP(\tau_b\in ds)
\end{align*}
To prove the upper bound we use the fact that $\PPP( B_{t_j-s} < 0, \ldots, B_{t_n-s} < 0)$ is a non-increasing function of $s \in [t_{j-1},t_j]$ (see Appendix \ref{appendix:grid_transformations}, Transformation T\ref{pb}), so that
\begin{align*}
w(b)\,\beta_T(b) & \leq \sum_{j=1}^n \int_{t_{j-1}}^{t_j} \PPP\Big( B_{t_j-t_{j-1}} < 0, \ldots, B_{t_n-t_{j-1}} < 0 \Big)\,\PPP(\tau_b\in ds) \\
& = \sum_{j=1}^n \PPP\Big( B_{t_j-t_{j-1}} < 0, \ldots, B_{t_n-t_{j-1}} < 0 \Big) \cdot \PPP\big(\tau_b\in (t_{j-1},t_j]\big).
\end{align*}
Dividing both sides of the inequality by $w(b) = \PPP(\tau_b\in(0,1])$ gives $\beta_T(b) \leq \bar{\beta}_T(b)$. To prove the lower bound we use Result \ref{appendix:results}.\ref{grid_ineq} from the Appendix, so as to obtain
\begin{align}
\nonumber w(b)\,\beta_T(b) & = \sum_{j=1}^n \int_{t_{j-1}}^{t_j} \PPP\Big( B_{t_j-s} < 0, \ldots, B_{t_n-s} < 0 \Big)\,\PPP(\tau_b\in ds) \\
\label{przejscieLB}  & \geq \sum_{j=1}^{n-1} \int_{t_{j-1}}^{t_j} \frac{1}{2}\,\PPP\Big( B_{t_{j+1}-t_{j}} < 0, \ldots, B_{t_n-t_{j}} < 0 \Big)\,\PPP(\tau_b\in ds) + \frac{1}{2}\,\PPP\big(\tau_b\in (t_{n-1},t_n]\big)  \\
\nonumber & \geq \sum_{j=1}^{n-1} \frac{1}{2}\,\PPP\Big( B_{t_{j+1}-t_j} < 0, \ldots, B_{t_n-t_{j}} < 0 \Big) \cdot \PPP\big(\tau_b\in (t_{j-1},t_j]\big) + \frac{1}{2}\,\PPP\big(\tau_b\in (t_{n-1},t_n]\big).
\end{align}
Dividing both sides of the inequality by $w(b)$ leads to $\beta_T(b) \geq \underbar{$\beta$}_T(b)$ and concludes the proof.
\end{proof}

\begin{proof}[Proof of Lemma \ref{equi_grid_increasing_b}] Recall the definitions of $a_j(b)$ and $w_j(b)$, and $\bar\beta_T(b) := \sum_{j=1}^n a_j(b) w_j(b)$. Notice that if we put $t_k = \frac{k}{n}$, then by the scaling property of Brownian Motion \[a_j(b) = \PPP\big( B_{1} < 0, \ldots, B_{1 + n-j} < 0 \big)\] and thus $a_1 < a_2 < \ldots < a_n$ (since the $a_j(b)$\,s are independent of $b$, we abbreviate $a_j := a_j(b)$).\\
\\
Assume that for any $0 < b_1 < b_2$ there exists $k \in \{1,\ldots,n-1\}$ such that
\begin{align}\label{to_prove_funnylemma2}
w_j(b_1) \geq w_j(b_2), \text{ for } j\leq k \ \ \ \text{ and } \ \ \ w_j(b_1) \leq w_j(b_2), \text{ for } j> k.
\end{align}
Since the weights $w_j(b)$ must satisfy $\sum_{j=1}^n w_j(b) = 1$ we have $\sum_{j=1}^n \big(w_j(b_2) - w_j(b_1)\big) = 0$ and thus \[\sum_{j=k+1}^{n} \big(w_j(b_2) - w_j(b_1)\big) = \sum_{j=1}^{k} \big(w_j(b_1) - w_j(b_2)\big).\] Finally,
\begin{align*}
\bar\beta_T(b_2) - \bar\beta_T(b_1) & = \sum_{j=1}^n a_j \big(w_j(b_2) - w_j(b_1)\big) = \sum_{j=k+1}^n a_j \big(w_j(b_2) - w_j(b_1)\big) - \sum_{j=1}^{k} a_j \big(w_j(b_1) - w_j(b_2)\big) \\
& \geq a_{k+1} \sum_{j=k+1}^n \big(w_j(b_2) - w_j(b_1)\big) - a_{k} \sum_{j=1}^{k} \big(w_j(b_2) - w_j(b_1)\big) \\
& = \big(a_{k+1} - a_k\big) \sum_{j=k+1}^n \big(w_j(b_2) - w_j(b_1)\big) > 0.
\end{align*}
For the remainder of the proof we prove the existence of $k\in\{1,\ldots,n-1\}$ satisfying (\ref{to_prove_funnylemma2}). Let $\tau_b := \inf\{t\geq 0 : B_t \geq b\}$ be the first hitting time of level $b$ and let $f(b,t)$ be the density of $\tau_b$ given that $\tau_b \leq 1$, i.e.,
\begin{align*}
f(b,t) := \frac{b}{2\Phi(-b)}t^{-3/2} \phi\left(-\frac{b}{\sqrt{t}}\right),
\end{align*}
where $b>0$, $t\in(0,1)$, and $\phi(\cdot)$ denotes the density of a standard normal random variable. We will prove that for any $0 < b_1 < b_2$ there exists $t^*$ such that:
\begin{align}\label{funny_property}
f(b_1,t) > f(b_2,t), \text{ for } t\in(0,t^*) \ \ \ \text{ and } \ \ \ f(b_1,t) < f(b_2,t), \text{ for } t\in(t^*,1].
\end{align}
Then the weights \[w_j(b) := \int_{t_{j-1}}^{t_j} f(b,t) \,dt\] are decreasing for all $\frac{j}{n}\leq t^*$, and increasing for all $\frac{j}{n}\geq \frac{1}{n}+t^*$. If $nt^*$ is not an integer, it is not known whether $w_{[t^*n] +1}(b)$ increases or not, but for sure there exists $k \in \{1,\ldots,n-1\}$ satisfying (\ref{to_prove_funnylemma2}). For the remainder we prove the existence of $t^*$ satisfying (\ref{funny_property}). For $t\in(0,1)$:
\begin{align*}
f(b_1,t) - f(b_2,t) & = \frac{1}{2}\,t^{-3/2} \, \left( \frac{b_1\,\phi\left(-\frac{b_1}{\sqrt{t}}\right)}{\Phi(-b_1)}  - \frac{b_2\,\phi\left(-\frac{b_2}{\sqrt{t}}\right)}{\Phi(-b_2)} \right) \\
& = \frac{1}{2}\,t^{-3/2} \frac{b_2\,\phi\left(-\frac{b_2}{\sqrt{t}}\right)}{\Phi(-b_2)} \, \left( \frac{b_1\,\phi\left(-\frac{b_1}{\sqrt{t}}\right)\Phi(-b_2)}{b_2\,\phi\left(-\frac{b_2}{\sqrt{t}}\right)\Phi(-b_1)}  - 1 \right) \\
& = \underbrace{\frac{1}{2}\,t^{-3/2} \frac{b_2\,\phi\left(-\frac{b_2}{\sqrt{t}}\right)}{\Phi(-b_2)}}_{> 0} \, \underbrace{\left( e^{\frac{b_2^2-b_1^2}{2t}}\,\frac{b_1\,\Phi(-b_2)}{b_2\,\Phi(-b_1)}  - 1 \right)}_{=:g(t)}
\end{align*}
Note that $\lim_{t\to0^+}g(t) = + \infty$ and $g(1) < 0$ (for example due to the Result \ref{appendix:results}.\ref{res1} in the Appendix) and that $g(\cdot)$ is strictly decreasing, hence $g(\cdot)$ has exactly one zero $t^*$ and $g(t) > 0$ for $t<t^*$ and $g(t) < 0$ for $t>t^*$. The observation that $\text{sign}(f(b_1,t) - f(b_2,t)) = \text{sign}(g(t))$ concludes the proof.
\end{proof}

\begin{proof}[Proof of Lemma \ref{lemma2_me}]
Let $h := \max_{k=1,\ldots,n}(t_k - t_{k-1})$. We transform the grid $T = \{t_1, \ldots, t_n\}$ with Transformations T\ref{pa}--T\ref{pc}, see Appendix \ref{appendix:grid_transformations}, in such a way that after all transformations we end up with $\{h, \ldots, Nh\}$.
\begin{enumerate}
\item Using Transformation T\ref{pb}, translate the grid to the right by $h-t_1$, i.e., put
\begin{align*}
t_j := t_j + h-t_1 \text { \ \ for all } j\in\{1,\ldots,n\}
\end{align*}
\item Put $\sigma_1 := 1$, $c_1 = 1$ and $k := 2$. While $k \leq N$ do:
\begin{itemize}
\item Put $\sigma_k := \inf\{j : t_j \geq kh\}$.
\item Using Transformation T\ref{pc}, contract the grid after time $t_{\sigma_{k-1}}$ by a factor $c_k$, where $c_k$ is defined by  ${h}/({t_{\sigma_k}-t_{\sigma_{k-1}}})$. Formally, we put
\begin{align*}
t_j := \begin{cases}
t_j, & j \in \{1,\ldots,\sigma_{k-1}\} \\
t_{\sigma_{k-1}} + c_k(t_j - t_{\sigma_{k-1}}), & j \in \{\sigma_{k-1}+1, \ldots, n\}
\end{cases}
\end{align*}
Notice that after this operation $t_{\sigma_k} = kh$.
\item Put $k:=k+1$.
\end{itemize}
\item Using Transformation T\ref{pa}, delete all the points $t_k$ such that $t_k \not\in\{h, \ldots, hN\}$.
\end{enumerate}
Now we prove that the algorithm is well-defined, more precisely, we confirm that all $\sigma_k$'s exist. First, see that $\sigma_1$ is well-defined. By induction, assume that $\sigma_k$ is well-defined and prove that $\sigma_{k+1}$ is well-defined as well. Notice that after the $k$th loop in Step 2 of the algorithm, the distances between the points shrunk at most by a factor $p_k = \prod_{j=1}^k c_j$ compared with the initial maximal distance $h$. Moreover, we observe that
\begin{align}\label{ckplus1}
c_{k} = \frac{h}{t_{\sigma_{k}}-t_{\sigma_{k-1}}} \geq \frac{h}{h + (t_{\sigma_{k}} - t_{\sigma_{k}-1})} \geq \frac{h}{h+\max_{j>\sigma_{k-1}} | t_{j+1} - t_{j}|} \geq \frac{1}{1 + \prod_{j=1}^{k-1} c_j}
\end{align}
We prove by induction that $p_k = \prod_{j=1}^k c_j \geq \frac{1}{k}$ for all $k\in\{1,\ldots,N\}$. Obviously $p_2 = c_2 \geq \frac{1}{2}$. Assume that $p_{k-1} \geq \frac{1}{k-1}$. After multiplying inequality (\ref{ckplus1}) by $p_{k-1}$ we obtain
\begin{align*}
p_{k} \geq \frac{p_{k-1}}{1 + p_{k-1}} = 1 - \frac{1}{1+p_{k-1}} \geq 1 - \frac{1}{1+\frac{1}{k-1}} = \frac{1}{k},
\end{align*}
which ends the inductive proof. Next, in order to show that $\sigma_{k+1}$ is well defined for $k\in\{1, \ldots N-1\}$ it suffices to prove that the endpoint $t_n$, after the $k$th loop of Step 2, is greater than $h(k+1)$. We prove a stronger statement, namely that the endpoint $t_n$ after being shrunk by a factor $p_k$ is still greater than $h(k+1)$, i.e. $h(k+1) \leq t_np_k$. By the definition of $N$, $h$ satisfies the inequality $h \leq {t_n}/{N^2}$, thus
\begin{align*}
h(k+1) \leq \frac{t_n(k+1)}{N^2} = \frac{t_n(k+1)}{N^2p_k}p_k = \frac{k(k+1)}{N^2} t_np_k \leq t_np_k,
\end{align*}
which concludes the proof that $\sigma_{k+1}$ is well-defined. As all transformations used in steps 1-3 satisfy (\ref{non-descreasity_of_n-grid_transformations}) we have  
\begin{align*}
\PPP\big(B_{t_1} > 0, \ldots, B_{t_n} > 0\big) \leq \PPP\big(B_h > 0, \ldots, B_{Nh} > 0\big)
\end{align*}
We finish the proof by observing that $\PPP\big(B_h > 0, \ldots, B_{Nh} > 0\big) = \PPP\big(B_1 > 0, \ldots, B_N > 0\big)$, due to the scaling property of Brownian Motion.
\end{proof}

\begin{proof}[Proof of Lemma \ref{lemming}] Notice that the grid points $t^n_k(b)$ defined in (\ref{THE_grid1}) depend only on the threshold $b$ and the \textit{ratio} $\frac{k}{n} \in [0,1]$. We are able to extend the definition of $t^n_k(b)$ to $t:(0,1]\times(0,\infty)\to[0,1]$,
\begin{align*}
t(s,b) := \Bigg(\frac{b}{\Phi^{-1}\big(s\,\Phi(-b)\big)} \Bigg)^{2}
\end{align*}
such that $t^n_k(b) = t(\frac{k}{n},b)$. Equivalently, $t(s,b)$ can be defined as the unique solution to
\begin{align}\label{general_tsb}
\Phi\left(-\frac{b}{\sqrt{t(s,b)}}\right) = s\,\Phi\left(-b\right)
\end{align}
This extension makes it possible to inspect the derivative of $t^n_k(b)$ with respect to the ratio $\frac{k}{n}$. Using the extension function of $t^n_k(b)$, we aim to prove the more general statement that for $0<s_1<s_2<1$,
\begin{align}\label{tobeequiv}
\frac{1 - t(s_1,b)}{t(s_2,b) - t(s_1,b)} \geq \frac{-\log s_1}{\log s_2 - \log s_1} \ \iff \ \ \frac{1 - t(s_1,b)}{-\log s_1} \leq \frac{1 - t(s_2,b)}{-\log s_2}.
\end{align}
Moreover, using the definition (\ref{general_tsb}) we may substitute \[s = \Phi\Big(-\frac{b}{\sqrt{t(s,b)}}\Big)/\Phi\left(-b\right)\] and arrive at another equivalent form of inequality (\ref{tobeequiv}):
\begin{align}\label{lhsfuntolemming}
\frac{1 - t(s_1,b)}{\log\big( \Phi(-b) \big) - \log\big( \Phi(-b/\sqrt{t(s_1,b)})\big)} \leq \frac{1 - t(s_2,b)}{\log\big( \Phi(-b) \big) - \log\big( \Phi(-b/\sqrt{t(s_2,b)})\big)} \, ,
\end{align}
which is Result \ref{appendix:results}.\ref{superresult} in the Appendix. For part (b) see that the density of the first hitting time,
\begin{align*}
\PPP(\tau_b\in ds) = \frac{b}{\sqrt{2\pi}}s^{-3/2}e^{-b^2/(2s)}\,ds, \text{ \ for } s>0
\end{align*}
is an increasing function on the interval $s\in[0,\frac{b^2}{3}]$ and thus part (b) follows from the second definition of the grid points $t^n_k(b)$ in (\ref{THE_grid2}).
\end{proof}

\appendix

\section{Grid transformations}\label{appendix:grid_transformations}

Let $T = \{t_1, \ldots, t_n\}$ with $0 < t_1 < \ldots < t_n < \infty$. We introduce three \textit{grid transformations}, i.e. operations $T \mapsto \widetilde{T}$ satisfying
\begin{align}\label{non-descreasity_of_n-grid_transformations}
\PPP\big(B_t > 0 \text{ for all } t\in T\big) \leq \PPP\big(B_t > 0 \text{ for all } t\in \widetilde{T}\big).
\end{align}
\begin{enumerate}[(T1)]
\item\label{pa} \textbf{Deleting}. For any $k\in\{1,\ldots,n\}$
\begin{align*}
\PPP\Big(B_{t_1} > 0, \ldots, B_{t_n} > 0\Big) \leq \PPP\Big(B_{t_1} > 0,\ldots,B_{t_{k-1}} > 0,B_{t_{k+1}} > 0,\ldots,B_{t_n} > 0\Big)
\end{align*}
\item\label{pb} \textbf{Translation to the right of the whole sequence}. For any $s>0$
\begin{align*}
\PPP\Big(B_{t_1} > 0, \ldots, B_{t_n} > 0\Big) \leq \PPP\Big(B_{t_1+s} > 0, \ldots, B_{t_n+s} > 0\Big)
\end{align*}
\item\label{pc} \textbf{Contraction of time after some point}. For any $k\in\{1,\ldots,n-1\}$ and $c\in(0,1)$:
\begin{align*}
\PPP\Big(B_{t_1} > 0, \ldots, B_{t_n} > 0\Big) \leq \PPP\Big(B_{t_1} > 0, \ldots,B_{t_k} > 0, B_{t_k + c(t_{k+1}-t_k)} > 0,\ldots,B_{t_k + c(t_n-t_k)} > 0\Big)
\end{align*}
\end{enumerate}

\begin{proof}[Proof that Transformations {\rm T1-T3} satisfy (\ref{non-descreasity_of_n-grid_transformations})]
Assertion T\ref{pa} is straightforward to verify. Observe for T\ref{pb} that
\begin{align*}
\PPP\big(B_{t_1} > 0, \ldots, B_{t_n} > 0\big) & = \int_0^\infty \PPP\big(B_{t_2} > 0, \ldots, B_{t_n} > 0 \mid B_{t_1} = x\big)\frac{1}{\sqrt{2\pi t_1}} e^{-x^2/(2t_1)}\,dx \\
& = \int_0^\infty \PPP\big(B_{t_2-t_1} < x, \ldots, B_{t_n-t_1} < x\big)\frac{1}{\sqrt{2\pi t_1}} e^{-x^2/(2t_1)}\,dx \\
& = \int_0^\infty \PPP\big(B_{t_2-t_1} < y\sqrt{t_1}, \ldots, B_{t_n-t_1} < y\sqrt{t_1} \big)\frac{1}{\sqrt{2\pi}} e^{-y^2/2}\,dy \\
& \leq \int_0^\infty \PPP\big(B_{t_2-t_1} < y\sqrt{t_1+s}, \ldots, B_{t_n-t_1} < y\sqrt{t_1+s} \big)\frac{1}{\sqrt{2\pi}} e^{-y^2/2}\,dy \\
& = \int_0^\infty \PPP\big(B_{t_2-t_1} < x, \ldots, B_{t_n-t_1} < x\big)\frac{1}{\sqrt{2\pi (t_1+s)}} e^{-x^2/(2(t_1+s))}\,dx \\
& = \PPP\big(B_{t_1+s} > 0, \ldots, B_{t_n+s} > 0\big)
\end{align*}
and for T\ref{pc} that
\begin{align*}
& \PPP\big(B_{t_1} > 0, \ldots, B_{t_n} > 0\big) \\
& = \int_0^\infty \PPP\big(B_{t_1} > 0, \ldots, B_{t_{k-1}} > 0 \mid B_{t_k} = x\big)\, \PPP\big(B_{t_{k+1}} > 0, \ldots, B_{t_n} > 0 \mid B_{t_k} = x\big)\frac{1}{\sqrt{2\pi t_k}} e^{-x^2/(2t_k)}\,dx \\
& = \int_0^\infty \PPP\big(B_{t_1} > 0, \ldots, B_{t_{k-1}} > 0 \mid B_{t_k} = x\big)\, \PPP\big(B_{t_{k+1}-t_k} < x, \ldots, B_{t_n-t_k} < x \big)\frac{1}{\sqrt{2\pi t_k}} e^{-x^2/(2t_k)}\,dx \\
& \leq \int_0^\infty \PPP\big(B_{t_1} > 0, \ldots, B_{t_{k-1}} > 0 \mid B_{t_k} = x\big)\, \PPP\left(B_{t_{k+1}-t_k} < \frac{x}{\sqrt{c}}, \ldots, B_{t_n-t_k} < \frac{x}{\sqrt{c}} \right)\frac{1}{\sqrt{2\pi t_k}} e^{-x^2/(2t_k)}\,dx \\
& = \int_0^\infty \PPP\big(B_{t_1} > 0, \ldots, B_{t_{k-1}} > 0 \mid B_{t_k} = x\big)\, \PPP\big(B_{c(t_{k+1}-t_k)} < x, \ldots, B_{c(t_n-t_k)} < x \big)\frac{1}{\sqrt{2\pi t_k}} e^{-x^2/(2t_k)}\,dx \\
& = \PPP\big(B_{t_1} > 0, \ldots,B_{t_k} > 0, B_{t_k + c(t_{k+1}-t_k)} > 0,\ldots,B_{t_k + c(t_n-t_k)} > 0\big)
\end{align*}
\end{proof}

\section{Miscellaneous results}\label{appendix:results}
Let $\Phi(\cdot)$ denote the standard normal cumulative distribution function and $\phi(\cdot)$ the standard normal density function. Below we list various results that we use. Results \ref{appendix:results}.\ref{ineq1}--\ref{appendix:results}.\ref{ineq2} are standard, and not proven here.
\begin{enumerate}[(\ref{appendix:results}.I)]
\item\label{ineq1} For $x>0$:
\begin{align}
\frac{x}{1+x^2} \leq \frac{\Phi(-x)}{\phi(x)} \leq \frac{1}{x}
\end{align}
\item\label{limit1} As $x\to\infty$,
\[\lim_{x\to\infty} \frac{\Phi(-x)}{\frac{1}{x}\phi(x)} \to 1.\]
\item\label{ineq2} \cite{szarek1999nonsymmetric}. For $x>-1$:
\begin{align}
\frac{2}{x + (x^2+4)^{1/2}} \leq \frac{\Phi(-x)}{\phi(x)} \leq \frac{4}{3x + (x^2+8)^{1/2}}
\end{align}
\item\label{grid_ineq} Let $0 < t_1 < \ldots < t_n < \infty$, then:
\begin{align*}
\PPP\Big(B_{t_1} > 0, \ldots, B_{t_n} > 0\Big) \geq \frac{1}{2}\,\PPP\Big(B_{t_2-t_1} > 0, \ldots, B_{t_n-t_1} > 0\Big)
\end{align*}
\item\label{mvt_taub} Let $T = \{t_1, \ldots, t_n\}$, where $t_j := \frac{j}{n}$, $\tau_b := \inf\{t\geq 0 : B_t \geq b\}$ and $b > 0$, then:
\begin{align*}
\PPP\Big(\tau_b\in (t_{j-1},t_j] \ \big| \ \tau_b\in(0,1]\Big) \leq \frac{b\,(b + \sqrt{b^2 +4})}{4} \cdot \frac{\sqrt{n}}{(j-1)^{3/2}} e^{-\frac{b^2}{2} \cdot \left(\frac{n}{j}-1\right)}
\end{align*}
and
\begin{align*}
\PPP\Big(\tau_b\in (t_{j-1},t_j] \ \big| \ \tau_b\in(0,1]\Big) \geq \frac{b\,(3b + \sqrt{b^2 +8})}{8} \cdot \frac{\sqrt{n}}{j^{3/2}} e^{-\frac{b^2}{2} \cdot \left(\frac{n}{j-1}-1\right)}
\end{align*}
for $j \in \{2, \ldots n\}$.
\item\label{res_increasing} Let $f:(0,\infty)\times(0,1)\to(0,\infty)$ such that
\begin{align*}
f(b,x) := \frac{b}{\sqrt{1-x}}x^{-3/2}e^{-\frac{b^2}{2}(1/x-1)}
\end{align*}
Then $f(b,x)$ is an increasing function of $x$, when $b\geq 1$.
\item\label{res1} Let $f:(0,\infty)\to(0,\infty)$ such that
\begin{align*}
f(x) := \frac{\Phi(-x)}{\phi(x)}
\end{align*}
Then $f$ is a strictly decreasing function.
\item\label{res2} Let $f:(0,\infty)\to(0,\infty)$ such that
\begin{align*}
f(x) := \frac{\Phi(-x)}{\frac{1}{x}\phi(x)}
\end{align*}
Then $f$ is a strictly increasing function.
\item\label{superresult} Let $f:[0,1]\to[0,\infty)$ be such that
\begin{align*}
f(t):= \begin{cases}
0, & t=0; \\
\vspace{-4mm}\\
\frac{\displaystyle 1-t}{\displaystyle \log\big( \Phi(-b) \big) - \log\big( \Phi(-b/\sqrt{t})\big)}, & t\in(0,1); \\
\vspace{-4mm}\\
\frac{\displaystyle 2\,\Phi(-b)}{\displaystyle b\,\phi(b)}, & t = 1.
\end{cases}
\end{align*}
Then $f$ is continuous and increasing.
\end{enumerate}


\subsection{Proofs of results \ref{appendix:results}.\ref{grid_ineq}--\ref{appendix:results}.\ref{superresult}}

\begin{proof}[Proof of \ref{appendix:results}.\ref{grid_ineq}] The proof is very similar to the proofs from Appendix \ref{appendix:grid_transformations}. Note that
\begin{align*}
\PPP\Big(B_{t_1} > 0, \ldots, B_{t_n} > 0\Big) & = \int_0^\infty \PPP\Big(B_{t_2} > 0, \ldots, B_{t_n} > 0 \mid B_{t_1} = x \Big)\frac{1}{\sqrt{2\pi t_1}} e^{-x^2/(2t_1)}\,dx \\
& = \int_0^\infty \PPP\Big(B_{t_2-t_1} < x, \ldots, B_{t_n-t_1} < x \Big)\frac{1}{\sqrt{2\pi t_1}} e^{-x^2/(2t_1)}\,dx \\
& \geq \int_0^\infty \PPP\Big(B_{t_2-t_1} < 0, \ldots, B_{t_n-t_1} < 0 \Big)\frac{1}{\sqrt{2\pi t_1}} e^{-x^2/(2t_1)}\,dx \\
& = \frac{1}{2}\,\PPP\Big(B_{t_2-t_1} > 0, \ldots, B_{t_n-t_1} > 0\Big),
\end{align*}
which concludes the proof.
\end{proof}

\begin{proof}[Proof of \ref{appendix:results}.\ref{mvt_taub}]
Using the mean value theorem and monotonicity of $\phi(\cdot)$ on the negative half-line, we have $|\Phi(-x) - \Phi(-y)| \leq |x-y| \cdot \phi(-y)$ for $0<y<x$. Furthermore,
\begin{align*}
\PPP\Big(\tau_b\in (t_{j-1},t_j] \Big) = 2\Phi(-b/\sqrt{t_j}) - 2\Phi(-b/\sqrt{t_{j-1}}) \leq 2\,\left(\frac{b}{\sqrt{t_{j}}} - \frac{b}{\sqrt{t_{j}}}\right) \, \phi\left(\frac{b}{\sqrt{t_{j}}}\right)
\end{align*}
Thus, for $b > 0$ and $j\in\{2,\ldots,n\}$, after substituting $t_j = \frac{j}{k}$, the above combined with the inequality \ref{appendix:results}.\ref{ineq2} yield:
\begin{align*}
\PPP\Big(\tau_b\in (t_{j-1},t_j] \ \big| \ \tau_b\in(0,1]\Big) & \leq \frac{1}{\Phi(-b)}\,\left(\frac{b}{\sqrt{t_{j-1}}} - \frac{b}{\sqrt{t_{j}}}\right) \, \phi\left(\frac{b}{\sqrt{t_{j}}}\right) \\
& = \frac{b \sqrt{n}}{\sqrt{2\pi} \Phi(-b)} \, \frac{\sqrt{j}-\sqrt{j-1}}{\sqrt{(j-1)j}} e^{-b^2n/(2j)} \\
& \leq \frac{b \sqrt{n}}{2\sqrt{2\pi} \Phi(-b)} \, \frac{1}{(j-1)^{3/2}} e^{-b^2n/(2j)} \\
& \leq \frac{b\,(b + \sqrt{b^2 +4})}{4} \, \frac{\sqrt{n}}{(j-1)^{3/2}} e^{-\frac{b^2}{2} \, \left(\frac{n}{j}-1\right)}
\end{align*}
The proof of the second inequality is analogous.
\end{proof}

\begin{proof}[Proof of \ref{appendix:results}.\ref{res_increasing}] It suffices to prove that $\frac{d}{dx}f(b,x) \geq 0$ for $b\geq 1$. See that
\begin{align*}
\frac{d}{dx}f(b,x) & =  be^{b^2/2} \, \frac{d}{dx}\, \frac{e^{-b^2/(2x)}}{\sqrt{1-x}\, x^{3/2}} \\
& = \frac{be^{b^2/2}}{(1-x)x^3} \, \left( \frac{b^2}{2x^2}e^{-b^2/(2x)}\sqrt{1-x}\, x^{3/2} - e^{-b^2/(2x)} \, \left( -\frac{x^{3/2}}{2\sqrt{1-x}} + \frac{3}{2}\sqrt{x(1-x)} \right) \right) \\
& = \frac{be^{b^2/2\,(1-1/x)}}{2(1-x)^{3/2}x^{7/2}} \, \left( b^2(1-x) + x^2 - 3x(1-x) \right) \\
& = \underbrace{\frac{be^{b^2/2\,(1-1/x)}}{2(1-x)^{3/2}x^{7/2}}}_{>0} \, \big( \underbrace{4x^2-(b^2+3)x+b^2}_{=:g(x)} \big)
\end{align*}
Note that $g(x)$ has at most one root when $b\in[1,3]$ thus $g(x) \geq 0$ for $b\in[1,3]$. Moreover, when $b>3$, then $g'(x) = 8x - (b^2+3) < -1$ (for $x\in[0,1]$) thus $g(x)$ is strictly decreasing for $x\in[0,1]$. From the observation that $g(0) = b^2>0$ and $g(1) = 1 >0$ we conclude that $g(x)$ is nonnegative on the interval $[0,1]$ for $b\geq1$ and thus $\frac{d}{dx} f(b,x) \geq 0$, when $b\geq 1$.
\end{proof}

\begin{proof}[Proof of \ref{appendix:results}.\ref{res1}] We have that
\begin{align*}
f'(x) = \frac{-\phi(x) + \Phi(-x)\, x}{\phi(x)},
\end{align*}
thus $f'(x) \leq 0$ iff $-\phi(x) + \Phi(-x)\, x \leq 0$, which is equivalent to Result \ref{appendix:results}.\ref{ineq1}.
\end{proof}

\begin{proof}[Proof of \ref{appendix:results}.\ref{res2}] See that
\begin{align*}
f'(x) = \frac{\Phi(-x) - x\phi(x) + x^2\Phi(-x)}{\phi(x)},
\end{align*}
thus $f'(x) \geq 0$ iff $\frac{\Phi(-x)}{\phi(x)} \geq \frac{x}{1+x^2}$, which is an implication of the lower bound from result \ref{appendix:results}.\ref{ineq2}.
\end{proof}

\begin{proof}[Proof of \ref{appendix:results}.\ref{superresult}] It is easy to see that $\lim_{t\to0^+} f(t) = 0$. To see that $\lim_{t\to 1^-} f(t) = \frac{2\,\Phi(-b)}{b\,\phi(b)}$ we expand $\log\big(\Phi(-b/\sqrt{t})\big)$ in a series around $t_0 = 1$ and obtain
\begin{align*}
\log\big(\Phi(-b/\sqrt{t})\big) = \log\big(\Phi(-b)\big) + \frac{b\,\phi(b)}{2\Phi(-b)}(t-1) + o(t-1)
\end{align*}
Thus
\begin{align*}
\lim_{t\to1^-} f(t) = \lim_{t\to1^-} \frac{1-t}{\frac{b\,\phi(b)}{2\Phi(-b)}(1-t) + o(t-1)} = \frac{2\,\Phi(-b)}{b\,\phi(b)}.
\end{align*}
To prove that $f$ is increasing we study the first derivative. For $t\in(0,1)$:
\begin{align}
\nonumber\frac{d}{dt}\,f(t) & = \frac{- \log\left( \frac{\Phi(-b)}{\Phi(-b/\sqrt{t})}\right) + \frac{(1-t)\Phi(-b/\sqrt{t})\Phi(-b)}{\Phi(-b)} \cdot \frac{bt^{-3/2}}{2\Phi(-b/\sqrt{t})^2}\phi(-b/\sqrt{t}) }{\Big(\log\big( \Phi(-b) \big) - \log\big( \Phi(-b/\sqrt{t})\big)\Big)^2} \\
\label{nominator_line}& = \frac{ \log\left( \frac{\Phi(-b/\sqrt{t})}{\Phi(-b)}\right) + \frac{b(1-t)}{2t^{3/2}} \cdot \frac{\phi(-b/\sqrt{t})}{\Phi(-b/\sqrt{t})} }{\Big(\log\big( \Phi(-b) \big) - \log\big( \Phi(-b/\sqrt{t})\big)\Big)^2}
\end{align}
Due to Result \ref{appendix:results}.\ref{res1} we have the lower bound \[\frac{\phi(-b/\sqrt{t})}{\Phi(-b/\sqrt{t})} \geq \frac{\phi(-b)}{\Phi(-b)}\] and thus the numerator of the fraction in (\ref{nominator_line}) can be bounded from below by the function $g:(0,1)\to\R$ defined as below:
\begin{align*}
g(t) := \log\left( \frac{\Phi(-b/\sqrt{t})}{\Phi(-b)} \right) + \frac{b(1-t)}{2\,t^{3/2}} \, \frac{\phi(b)}{\Phi(-b)}
%
\end{align*}
Notice that $g(t)\geq 0$ implies $\frac{d}{dt}f(t)\geq 0$ which is exactly what we want to establish. For the remainder of the proof we show that $g(t)$ is non-negative. Since $\lim_{t\to0^+}g(t) = +\infty$ and $g(1) = 0$, it suffices to show that $g'(t)$ is monotone (non-increasing). We study the first derivative
\begin{align*}
g'(t) & = \frac{b}{2t^{3/2}} \frac{\phi(b/\sqrt{t})}{\Phi(-b/\sqrt{t})} + \frac{b}{4t^{3/2}} \frac{\phi(b)}{\Phi(-b)} - \frac{3b}{4t^{5/2}} \frac{\phi(b)}{\Phi(-b)} \\
& = \frac{b^2}{4t^2} \left( 2\, \frac{\frac{\sqrt{t}}{b}\phi(b/\sqrt{t})}{\Phi(-b/\sqrt{t})} + \left(t^{1/2} - 3t^{-1/2} \right) \, \frac{\frac{1}{b}\phi(b)}{\Phi(-b)}\right) \\
& \leq \frac{b^2}{4t^2} \left( 2\, \frac{\frac{\sqrt{t}}{b}\phi(b/\sqrt{t})}{\Phi(-b/\sqrt{t})} -2 \, \frac{\frac{1}{b}\phi(b)}{\Phi(-b)}\right) \leq 0,
\end{align*}
where the last inequality is a consequence of the application of Result \ref{appendix:results}.\ref{res2}, that is \[t \longmapsto \frac{\frac{\sqrt{t}}{b}\phi(b/\sqrt{t})}{\Phi(-b/\sqrt{t})}\] is an increasing function of $t$.
\end{proof}

\medskip

\noindent \textbf{Acknowledgments}. 
\red{The authors would like to thank Ankush Agarwal and Johan van Leeuwaarden for useful discussions and suggestions.
In addition, the reviewers' reports helped improving the quality of our work considerably.}
This work is part of the research programme `Rare Event Simulation for Climate Extremes' with grant number 657.014.033, which is (partly) funded by the Netherlands Organisation for Scientific Research (NWO). Michel Mandjes' research is partly funded by the NWO Gravitation Programme NETWORKS, grant number 024.002.003.

\medskip
\bibliographystyle{apalike}
\bibliography{bibliografia_optimal_grids}
\end{document}